\documentclass[11pt,a4paper,reqno]{amsart}
\usepackage[margin=2.5cm]{geometry}
\usepackage{color}               
\usepackage{faktor}
\usepackage{tikz-cd}
\usepackage{enumitem}
\usepackage{hyperref}

\usepackage{amsmath,amssymb,anyfontsize,stmaryrd,mathtools, thmtools,tikz}
\usepackage{mathrsfs}
\usepackage[oldstyle]{libertine}%
\usepackage[T1]{fontenc}
\usepackage[utf8]{inputenc}
\usepackage{tikz-cd}
\usepackage[noabbrev]{cleveref}
\usepackage{autonum}
\usepackage{rotating}

\usepackage{url}
\crefname{lemma}{lemma}{lemmata}
\Crefname{lemma}{Lemma}{Lemmata}
\crefname{subsection}{subsection}{subsections}
\Crefname{subsection}{Subsection}{Subsections}

\newtheorem{theorem}{Theorem}[section]
\newtheorem{lemma}[theorem]{Lemma}
\newtheorem{proposition}[theorem]{Proposition}
\newtheorem{corollary}[theorem]{Corollary}
\newtheorem{conjecture}[theorem]{Conjecture}

\newtheorem*{theorem*}{Theorem}
\theoremstyle{definition}
\newtheorem{definition}[theorem]{Definition}
\newtheorem{remark}[theorem]{Remark}

\newtheorem{example}[theorem]{Example}

\newcommand{\mb}[1]{\mathbb{#1}}

\newcommand{\C}{\mb{C}} 
\newcommand{\Z}{\mb{Z}}
\newcommand{\fG}{\mathfrak{G}} 
\newcommand{\fI}{\mathfrak{I}} 
\newcommand{\cA}{\mathcal{A}}

  \newcommand{\rk}{\mathrm{rank}}

\newcommand{\Ch}{{\mathfrak{ C}}}
\newcommand{\Oh}{{\mathcal O}}

\newcommand{\Syz}{\mathrm{Syz}}

\newcommand{\cE}{\mathcal{E}}

\newcommand\restr[2]{{
  \left.\kern-\nulldelimiterspace 
  #1 
  \vphantom{\big|} 
  \right|_{#2} 
  }}

\newcommand\restrst[2]{{
  \left.\kern-\nulldelimiterspace 
  #1 
  \vphantom{\big|} 
  \right|^*_{#2} 
  }}

\begingroup\expandafter\expandafter\expandafter\endgroup
\expandafter\ifx\csname pdfsuppresswarningpagegroup\endcsname\relax
\else
\fi
{}

\title{Mustafin models of projective varieties and vector bundles}
\author[M.~A.~Hahn]{Marvin Anas Hahn}
\address{M.~A.~Hahn: Institut für Mathematik, Goethe-Universität Frankfurt, Robert-Mayer-Str. 6-8, 60325 Frankfurt am Main}
\email{hahn@math.uni-frankfurt.de}


\begin{document}
\begin{abstract}
Mustafin varieties are well-studied degenerations of projective spaces induced by a choice of integral points in a Bruhat--Tits building. In recent work, Annette Werner and the author initiated the study of degenerations of plane curves obtained by Mustafin varieties by means of arithmetic geometry. Moreover, we applied these techniques to construct models of vector bundles on plane curves with strongly semistable reduction. In this work, we take a Groebner basis approach to the more general problem of studying degenerations of projective varieties. Our methods include determining the behaviour of Groebner bases under substitution over unique factorisation rings. Finally, we outline applications to the $p-$adic Simpson correspondence, when the respective projective variety is a curve.
\end{abstract}
\maketitle

\section{Introduction}
In this paper, we study degenerations of projective varieties induced by point configurations in the Bruhat-Tits buildings associated to $\mathrm{GL}(V)$. At the heart of our considerations are so-called \textit{Mustafin varieties}. Mustafin varieties were first introduced by Mustafin in \cite{mustafin1978nonarchimedean} in order to generalise Mumford's seminal work on the uniformisation of curves \cite{mumford1972analytic}. Since then, they have attracted a lot of interest. In particular, they have found applications in the context of Shimura varieties \cite{fa} and Chow quotients of Grassmannians \cite{zbMATH05051634}. In recent years, especially their interesting feature that their reductions may be described in terms of the combinatorics of convex hulls in Bruhat--Tits buildings has been studied. A conceptual combinatorial framework for this perspective was developed in \cite{cartwright2011mustafin} and used in \cite{hahn2017mustafin} to establish relations between Mustafin varieties, computer vision and limit linear series. Applications to limit linear series were further investigated in \cite{he2019linked}. A generalisation to degenerations of flag varieties was introduced in \cite{habich2014mustafin}.\\
Another interesting feature is that one may choose global coordinates on Mustafin varieties, which enables the use of computer algebra techniques to study these degenerations.\\
For any projective embedding of a projective variety $X$, any Mustafin variety of the ambient projective space induces a degeneration $\mathcal{X}$ of $X$ by restriction. We call the obtained model a \textit{Mustafin degeneration} of $X$. Such models have first been studied in recent work of Werner and the author, when $X$ is a curve embedded into the plane \cite{HWplane}. The main motivation for this work stemmed from recent advances towards a $p-$adic Simpson correspondence. The classical Simpson correspondence in dimension one establishes a correspondence between semistable degree zero Higgs bundles on a Riemann surface $X$ and representations of its topological fundamental group  \cite{simpson1990nonabelian,simpson1992higgs}. In recent years, the question whether a similar result holds in the $p-$adic setting has developed to a major research theme in arithmetic geometry \cite{dewe1,dewe4,fa,abbes2016p, liuzhu}. In \cite{fa}, Faltings proved an equivalence of categories between Higgs bundles on a $p-$adic curve and so-called \textit{generalised representations} of its étale fundamental group. These generalised representations contain the continuous representations of the \'{e}tale fundamental group as a full subcategory. The remaining challenge is to identify the subcategory of Higgs bundles, which is equivalent to the category of continuous representations. This is still an open problem. An approach for Higgs bundles with trivial Higgs field was introduced in \cite{dewe1, dewe4} by Deninger and Werner and shown to be compatible with Faltings' functor in \cite{xu} (see also \cite{wurthen2019vector}). More precisely, it is shown that a semistable vector bundle on a proper, smooth $p$-adic curve $X$ which has strongly semistable reduction of degree zero after pullback to a finite covering of the curve  admits $p$-adic parallel transport and hence gives rise to a continuous representation of the \'etale fundamental group. One is of course tempted to speculate that the desired category is that of semistable Higgs bundles of degree zero. By the results of Deninger and Werner a positive answer for semistable degree zero bundles with trivial Higgs fields can be given if we prove a potentially strongly semistable reduction theorem for all such bundles. This involves -- possibly after pull-back to a finite covering -- finding for any semistable vector bundle $E$ of degree zero on a projective curve $X$, a model $\mathcal{E}$ on a model of the curve $\mathcal{X}$, such that the special fibre of $\mathcal{E}$ is strongly semistable on each irreducible component. This is as might be expected a very difficult task.\\
In \cite{HWplane}, Werner and the author proposed to use Mustafin degenerations of projective curves to construct such models. This approach has proved to be quite succesful in \cite{HWplane}, as we were able to construct models of a class of syzygy bundles on projective curves with strongly semistable reduction. The models were flexible enough to refute a proposed counterexample, which was suggested by Brenner in \cite{br2}, to the above speculation that the desired category is that of semistable Higgs bundles of degree zero. The case of syzygy bundles is of particular importance since all vector bundles on projective curves may be realised as syzygy bundles after tensor product with a line bundle (\cref{lem-allsyz}). Moreover, it was shown in \cite{dewe1,dewe4} that potentially strongly semistable reduction is compatible with tensor products and that line bundles admit strongly semistable reduction. Therefore, a complete classification of syzygy bundles with potentially strongly semistable reduction also yields a complete classification of all vector bundles with potentially strongly semistable reduction.\\
The above discussion indicates that a conceptual framework for the study of Mustafin degenerations of projective varieties is needed. While the progress in \cite{HWplane} is promising, the methods rely on some intricate arithmetic geometry. In this paper, we instead propose a computer algebra approach towards this problem. More precisely, we generalise the results in \cite{HWplane} using only the theory of Groebner bases. This paper should therefore be seen as the proposal for a computer algebra programme towards  some open problems in the $p-$adic Simpson correspondence.\\
We begin by proposing a conjecture on the generic behaviour of a large class of Mustafin varieties (\cref{conj-must}). This conjecture seems closely related to \cite[conjecture 1.1]{conca2007linear} and the theory of generic initial ideals (\cref{rem-genin}). We prove \cref{conj-must} over several base fields in \cref{sec:first} in dimension $3$ by direct computations in \textsc{Singular} \cite{DGPS}; in particular over the $p-$adic numbers $\mathbb{Q}_p$ for $p\gg0$. In dimension $2$ over any field it was proved in \cite[Lemma 3.2]{HWplane}. Based on this conjecture, we formulate our main theorem in \cref{thm-main1} in which we completely determine the combinatorial structure of Mustafin degenerations of a projective variety, whenever the inducing point configuration is of the type described in \cref{conj-must}.\\
Our main tool in the proof of \cref{thm-main1} is the use of Groebner bases over unique factorisation domains. In particular, we give sufficient criteria when such Groebner bases remain Groebner bases after evaluating some of the variables. While when the base ring is a field, this is a well-studied topic which has led to the notion of a \textit{comprehensive Groebner basis} \cite{weispfenning1992comprehensive, weispfenning2003canonical}, the problem for more general rings is largely unexplored. Our results regarding this problem are therefore also of independent interest.\\
Finally, in \cref{sec:models} we use \cref{thm-main1} to construct models of syzygy bundles on any projective curve, which under the assumption of \cref{conj-must} admit strongly semistable reduction. Since any smooth projective curve may be realised in dimension $3$ and by the results in \cref{sec:first}, we have therefore produced models of vector bundles on any projective curve over $\mathbb{Q}_p$ with $p\gg0$ of arbitrary rank. Thus, for $p\gg0$ we provide new families of semistable vector bundles of degree $0$ on any smooth projective curve which fit into the framework of the $p-$adic Simpson correspondence.

\subsection*{Acknowledgements} The author is indebted to Annette Werner for many interesting discussions and comments  on this project. We also thank Holger Brenner for helpful comments and for making us aware of \cref{lem-allsyz}. The author gratefully acknowledge support of the LOEWE research unit Uniformized Structures in Arithmetic and Geometry. Many computations for this project were aided by \textsc{Singular} \cite{DGPS}.

\section{Preliminaries}
For the rest of this paper, let $K$ be a non-archimedean field with ring of integers $\mathcal{O}_K$, maximal ideal $\mathfrak{m}_K$ and residue field $k\cong\faktor{\mathcal{O}_K}{\mathfrak{m}_K}$. Further, let $\pi$ be a fixed uniformiser of $\mathcal{O}_K$.

\subsection{Mustafin varieties}
\label{sec:premust}
In this subsection, we discuss the basic notions revolving around \textit{Mustafin varieties} (see also \cite{cartwright2011mustafin,hahn2017mustafin}). Let $V$ be a vector space of dimension $d$ over $K$. Regarding $V$ as an $\mathcal{O}_K-$module, we call a free $\mathcal{O}_K$-submodule $L\subset V$ of rank $d$ a \textit{lattice}. We further define
\begin{equation}
\mathbb{P}(V)=\mathrm{Proj}\mathrm{Sym} V^\ast\quad\textrm{and}\quad \mathbb{P}(L)=\mathrm{Proj}\mathrm{Sym}\left(\mathrm{Hom}_{\mathcal{O}_K}(L,\mathcal{O}_K)\right).
\end{equation}
Mostly, we will consider homothety classes of lattices, i.e. we call two lattices $L,L'$ equivalent if $L=c\cdot L'$ for $c\in K^\ast$. We denote the homothety class of a lattice $L$ by $[L]$.

\begin{definition}
\label{def:musta}
Let $\Gamma=\{[L_0],\dots,[L_n]\}$ be a set of rank $d$ lattice classes in $V$. Then $\mathbb{P}(L_0),\dots,\allowbreak\mathbb{P}(L_n)$ are projective spaces over $\mathcal{O}_K$ whose generic fibres are canonically isomorphic to $\mathbb{P}(V)\simeq\mathbb{P}^{d-1}_{K}$. The open immersions
\begin{equation}
\mathbb{P}(V)\hookrightarrow\mathbb{P}(L_i)
\end{equation}
give rise to a map
\begin{equation}
\label{equ:mapmust}
f_{\Gamma}:\mathbb{P}(V)\longrightarrow\mathbb{P}(L_0)\times_{\mathcal{O}_K}\dots\times_{\mathcal{O}_K}\mathbb{P}(L_n).
\end{equation}
We denote the closure of the image endowed with the reduced scheme structure by $\mathcal{M}(\Gamma)$. We call $\mathcal{M}(\Gamma)$ the \textit{Mustafin variety associated to} $\Gamma$. Its special fibre $\mathcal{M}(\Gamma)_k$ is a reduced scheme over $k$ by \cite[Theorem 2.3]{cartwright2011mustafin}.\end{definition}

There is a natural way to choose coordinates on Mustafin varieties. For this, we fix a reference lattice $L=\mathcal{O}_Ke_1+\dots+\mathcal{O}_Ke_d$, where $e_1,\dots,e_d$ is the standard basis of $V$. For $L_0,\dots,L_n$ as in \cref{def:musta}, we can find $g_0,\dots,g_n\in\mathrm{PGL}(V)$ such that $g_iL=L_i$. We consider the commutative diagram
\begin{equation}
\begin{tikzcd}
\mathbb{P}(V)\arrow{rr}{(g_0^{-1},\dots,g_n^{-1})\circ\Delta} \arrow{d} &  & \mathbb{P}(V)^n\arrow{d}\\
\prod_{R}\mathbb{P}(L_i) \arrow{rr}{(g_1^{-1},\dots,g_n^{-1})} &  & \mathbb{P}(L)^n.
\end{tikzcd}
\end{equation}
Let $x_1,\dots,x_d$ be the coordinates on $\mathbb{P}(L)$ and consider the projections
\begin{equation}
P_j:\mathbb{P}(L)^n\to\mathbb{P}(L)
\end{equation}
to the $j-$th factor. Then, we denote $x_{ij}=P_j^*x_i$ and observe that the Mustafin variety $\mathcal{M}(\Gamma)$ is isomorphic to the subscheme of $\mathbb{P}(L)^n$ cut out by
\begin{equation}
\label{equ-idmust}
I_2\begin{pmatrix}
g_1\begin{pmatrix}
x_{10}\\
\vdots\\
x_{d0}
\end{pmatrix} &
\cdots
& g_n\begin{pmatrix}
x_{1n}\\
\vdots\\
x_{dn}
\end{pmatrix}
\end{pmatrix}\cap \mathcal{O}_K[(x_{ij})].
\end{equation}
By
\begin{equation}
p_j=\restr{P_j}{\mathcal{M}(\Gamma)}:\mathcal{M}(\Gamma)\hookrightarrow\mathbb{P}(L)^n\to\mathbb{P}(L)
\end{equation}
we denote the projection to the $j-$th component. We write $x_{ij}$ also for the induced rational function on $\mathcal{M}(\Gamma)$. By \cite[Corollary 2.5]{cartwright2011mustafin}, for each $i$ there exists a unique irreducible component $X$ of $\mathcal{M}(\Gamma)_k$ which maps birationally onto $\mathbb{P}(L)_k$ via the map on the special fibre induced by $p_i$. We call $X$ the $i-$\textit{th primary component} of $\mathcal{M}(\Gamma)_k$. Furthermore, we give the following definition.

\begin{definition}
\label{def-length}
Let $\Gamma=\{[L_0],\dots,[L_n]\}$ be a set of lattices and $\mathcal{M}(\Gamma)$ be its associated Mustafin variety and let $C\subset\mathcal{M}(\Gamma)_k$ be an irreducible component of the special fibre. Further, let $J_C\subset\{0,\dots,n\}$ be the maximal subset of $\{0,\dots,n\}$, such that we have $\mathrm{dim}(\mathrm{pr}_j(C))>0$. For $l\coloneqq|J:C|$, we call $C$ a component of length $l$ of $\mathcal{M}(\Gamma)_k$. We set $\mathcal{M}(\Gamma)_{k,\le l}$ the union of all irreducible components of length $\le l$. Moreover, we call call $J_C$ the \textit{support of} $C$ and denote $\mathrm{supp}(C)=J_C$.
\end{definition}

\subsection{Syzygy bundles}
We consider syzygy sheaves on the projective space which are the kernel of a morphism to the structure sheaf. To be precise, let $f_0, \ldots, f_{n}$ be homogeneous polynomials in $K[x_1,\dots,x_N]$ with degrees $d_0, \ldots, d_{n}$. Then the corresponding syzygy sheaf $\Syz(f_0, \ldots, f_{n})$ on $\mathbb{P}_K^{N-1}$ is defined as the kernel

\[ 0 \longrightarrow \Syz(f_0, \ldots, f_{n}) \longrightarrow \bigoplus_{i=0}^{n} \Oh(-d_i) \xrightarrow{(f_0, \ldots, f_{n})}\mathcal{O}.\]
The sheaf $\Syz(f_0, \ldots, f_{n})$ is locally free on $\bigcup D_+(f_i)$. 

In this work, we will be concerned with vector bundles of degree zero on curves. Therefore, we consider the twisted sheaves $\Syz(f_0, \ldots, f_{n})(\rho)$ when $\sum d_i=n\rho$.

\begin{remark}
We note that usually a coherent sheaf $\mathcal{F}$ on $X$ is called a $k-$\textit{th syzygy sheaf} if for each $x\in X$, there exist an open neighbourhood $U$ of $x$, locally free sheaves $\mathcal{G}_1,\dots,\mathcal{G}_k$ on $U$ and an exact sequence
\begin{equation}
0\to\restr{\mathcal{F}}{U}\to\mathcal{G}_1\to\dots\to\mathcal{G}_k.
\end{equation}
Thus the sheaf $\Syz(f_1, \ldots, f_{n+1})$ is a second syzygy sheaf.
\end{remark}

As mentioned in the introduction, all vector bundles on smooth projective curves may be realised as syzygy bundles after tensor product with an invertible sheaf. The following lemma and proof was communicated to us by Holger Brenner.

\begin{lemma}
\label{lem-allsyz}
Let $C$ be a smooth projective curve over an algebraically closed field and $E$ a vector bundle on $C$ of rank $r$. Then, there exists a line bundle $L$ on $X$, a polarisation $\mathcal{O}(1)$ of $X$ and $n\in\mathbb{Z}_{\ge1}$, such that we have an exact sequence
\begin{equation}
0\to E\otimes L\to \mathcal{O}^{r+1}\to \mathcal{O}(n)\to 0.
\end{equation} 
\end{lemma}

\begin{proof}
Let $E^\vee$ be the vector bundle dual to $E$. Let $M$ be a very ample line bundle on $C$, then there exists $l\in\mathbb{N}$, such that $N=\mathrm{det}(E^\vee\otimes M^{\otimes l})$ is very ample. We then set $F=E^\vee\otimes M^{\otimes l}$. Let $C$ be polarised by $N$. Then, there exists $m\in\mathbb{N}$, such that $F(m)=E^\vee\otimes M^{\otimes l}\otimes N^{\otimes m}$ is globally generated. As proved in \cite[lemma 2.3]{brenner2006bounds}, there exists a surjection
\begin{equation}
\mathcal{O}^{r+1}\to F(m).
\end{equation}
Then $L'=\mathrm{ker}(\mathcal{O}^{r+1}\to F(m))$ is a line bundle. We have the exact sequence
\begin{equation}
\label{equ-sequ1}
0\to L'\to \mathcal{O}^{r+1}\to F(m)\to 0.
\end{equation}
It is well-known that for a short exact sequence
\begin{equation}
0\to U\to V\to W\to 0
\end{equation}
of vector bundles on a smooth projective curve, we have $\mathrm{det}(U)\otimes\mathrm{det}(W)\cong\mathrm{det}(V)$. Thus, for \cref{equ-sequ1}, we obtain $\mathrm{det}(L')\otimes\mathrm{det}(F(m))=\mathcal{O}$. As $\mathrm{det}(L')=L'$ and $\mathrm{det}(F(m))=N^{\otimes(r+m)}=\mathcal{O}(r+m)$, we have $L'=\mathcal{O}(-(r+m))$. Therefore, we obtain an exact sequence
\begin{equation}
0\to \mathcal{O}(-(r+m))\to\mathcal{O}^{r+1}\to F(m)\to 0.
\end{equation}
After dualising, we obtain
\begin{equation}
0\to E \otimes (M^\vee)^{\otimes l}(-m)\to\mathcal{O}^{r+1}\to \mathcal{O}(r+m)\to 0.
\end{equation}
Setting $n=r+m$ and $L=(M^\vee)^{\otimes l}(-m)$, the lemma follows.
\end{proof}

\subsection{Semistability of vector bundles and parallel transport for $p$-adic vector bundles}
\label{sec:semi}
Recall that a vector bundle $E$  on  a smooth, projective and connected curve $C$ over a field $\kappa$ is semistable (respectively stable), if for all proper non-zero subbundles $F$ of $E$ the inequality $\deg (F)/\rk(F)  \le \deg (E)/\rk(E)$ (respectively $\deg (F)/\rk(F) < \deg (E)/\rk(E) $) holds.

If $\kappa$ has positive characteristic, semistability has weaker properties than in characteristic zero, since this property may be lost under pullback by inseparable morphism. This explains the following notion of strong semistability.

Assume that $\mbox{char} (\kappa) = p>0$, and let
$F : C \to C$ be the absolute Frobenius morphism, defined by the $p$-power map
on the structure sheaf. 
Then a vector bundle $E$ on $C$ is called strongly semistable, 
if $F^{n*} E$ is semistable on $C$ for all $n
\ge 1$.

\begin{definition}
\label{def:strong}
Let $E$ be a vector bundle on a one-dimensional proper scheme $C$ over a field $\kappa$ of characteristic $p$. Then $E$ is called strongly semistable of degree zero, if the pullback of $E$ to all normalized irreducible components of $C$ is strongly semistable of degree zero.
\end{definition}

Consider a smooth, projective and connected curve $C$ over $\overline{\mathbb{Q}}_p$, and denote by $C_{\C_p}$ the base change to the field $\C_p$ (which is the completion of the algebraic closure $\overline{\mathbb{Q}}_p$). By $\mathfrak{o}$ we denote the ring of integers of $\C_p$. Its residue field is isomorphic to $\overline{\mathbb{F}}_p$. We call every finitely presented, flat and proper $\overline{\Z}_p$-scheme $\mathfrak{C}$ with generic fibre $C$ a model of $C$.

\begin{definition}
\label{def:strongred}
A vector bundle $E$ on $C_{\C_p}$ has strongly semistable reduction of degree zero, if there exists a model $\Ch$ of $C$ and a vector bundle $\mathcal{E}$ on $\Ch_{\mathfrak{o}} = \Ch \otimes_{\overline{\mathbb{Z}}_p} \mathfrak{o}$ such that $\mathcal{E}$ has generic fibre $E_{\mathbb{C}_p}$ and such that the special fibre  $\mathcal{E}_{\overline{\mathbb{F}}_p}$ of $\mathcal{E}$ is strongly semistable of degree zero on the one-dimensional proper scheme $\Ch \otimes_{\overline{\mathbb{Z}}_p} \mathbb{F}_p$ in the sense of  \cref{def:strong}.
\end{definition}

In \cite{dewe1} and \cite{dewe3}, a theory of parallel transport along \'etale paths is defined for those vector bundles $E$ of degree zero on $C_{\C_p}$ for which there exists a finite, \'etale covering $\alpha: C' \rightarrow C$ such that the  bundle $\alpha_{\C_p}^\ast E$ on $C'_{\C_p}$ has strongly semistable reduction of degree zero. We note that if $E$ has strongly semistable reduction of degree zero, then $E$ is semistable of degree zero \cite[Theorem 13]{dewe1}.

\begin{definition}
\label{def:potstrongred}
A vector bundle $E$ of degree zero on $C_{\C_p}$ has \textit{potentially strongly semistable reduction} if there exists a finite (not necessarily \'etale) covering $\alpha: C' \rightarrow C$ such that the bundle $\alpha^\ast_{\C_p} E$ on $C'_{\C_p}$ has strongly semistable reduction.
\end{definition}

It is an important open question if all semistable bundles of degree zero on $C_{\C_p}$ have potentially strongly semistable reduction in this sense. In fact,  \cite[Theorem 10]{dewe4} implies that all bundles with potentially strongly semistable reduction admit $p$-adic parallel transport. Hence, using \cite{xu} and \cite{fa}, a positive answer to this question would imply that all semistable bundles of degree zero on $C_{\C_p}$ with trivial Higgs field correpond to $p$-adic representations of the \'etale fundamental group under the $p$-adic Simpson correspondence, which would represent a big step in the directon of a $p$-adic result which is analogous to the classical Simpson correspondence.\\
In \cite{br2}, an counter example for the claim that any semistable vector bundle of degree zero on $C_{\C_p}$ admits potentially strongly semistable reduction was proposed. In previous work \cite{HWplane}, the authors were able to refute this counter example the techniques we generalise in this work.

\subsection{Groebner bases over noetherian domains}
We now briefly recall the theory of Groebner bases over general noetherian domains. We will mostly be interested in Groebner basws for polynomial ideals with coefficients in a discrete valuation ring.

\begin{definition}
A total order $<$ on $\mathbb{N}^n$ is called a \textit{term order} if $\underline{0}=(0,\dots,0)$ is the minimal element and $\alpha<\beta$ implies $\alpha+\gamma<\beta+\gamma$ for all $\alpha,\beta,\gamma\in\mathbb{N}^n$.
\end{definition}

\begin{example}
An important example of a term order is the so-called \textit{lexicographic order}. We define $\alpha<\beta$ for $\alpha,\beta\in\mathbb{N}^n$ if for the smallest $j$, such that $\beta_j-\alpha_j\neq0$, we have $\beta_j-\alpha_j>0$.
\end{example}

Let $R$ be a commutative ring and define $A=R[x_1,\dots,x_n]$. As usual, we associate to $\alpha=(\alpha_1,\dots,\alpha_n)\in\mathbb{N}^{n}$ the monomial $x^{\alpha}=x_1^{\alpha_1},\dots,x_n^{\alpha_n}$.

\begin{definition}
Let $f=\sum_{\alpha\in\mathbb{N}^n}c_{\alpha}x^{\alpha}\in A$ and $<$ a monomial order on $\mathbb{N}^n$. We define
\begin{itemize}
\item $\mathrm{DEG}(f)=\mathrm{max}_<\{\alpha\mid c_{\alpha}\neq 0\}$,
\item $\mathrm{lm}(f)=x^{\mathrm{DEG}(f)}$,
\item $\mathrm{lc}(f)=c_{\mathrm{DEG}(f)}$,
\item $\mathrm{lt}(f)=c_{\mathrm{DEG}(f)}x^{\mathrm{DEG}(f)}$.
\end{itemize}
\end{definition}

We now define Groebner bases for ideals in $A$ in the usual sense.

\begin{definition}
\label{def-groe}
Let $E\subset A$ and fixed a term order $<$, then we define
\begin{equation}
\mathrm{Lt}_<(E)=\{lt(f)\mid f\in E\}.
\end{equation}
Let $0\neq I\subset A$ be an ideal and $G\subset I$ a finite generating set. We say $G$ is a Groebner basis of $I$ with respect to a term order $<$ if
\begin{equation}
\mathrm{Lt}_<(G)=\mathrm{Lt}_<(I).
\end{equation}
\end{definition}

The classical Buchberger criterion for Groebner bases has a similar analogue for arbitrary base rings. In order to state the modified criterion, we first introduce the following notion.

\begin{definition}
Let $E=\{g_1,\dots,g_m\}\subset A$ be a set of non-zero polynomials in $A$ and let $f,g\in A$ be polynomials. We say \textit{$f$ reduces to $g$ modulo $E$ in one step}, denoted by $f\xrightarrow{E} g$, if
\begin{itemize}
\item There exists at least one index $j\in\{1,\dots,s\}$, such that $\mathrm{lc}(f_j)$ divides $\mathrm{lc}(f)$.
\item For $J=\{j\mid \mathrm{lm}(f_j)\,\textrm{divides}\,\mathrm{lm}(f)\}$, there exists $(c_j)_{j\in J}$, such that
\begin{equation}
\sum_j c_j\mathrm{lc}(f_j)=\mathrm{lc}(f)
\end{equation}
and
\begin{equation}
h=f-\sum_{j\in J}c_j\frac{\mathrm{lm}(f)}{\mathrm{lm}(f_j)}\mathrm{lt}(f_j).
\end{equation}
\end{itemize}
We further say \textit{$f$ reduces to $g$ modulo $E$}, denoted by $f\xrightarrow{E}_+ g$, if there exist $h_1,\dots,h_s\in A$, such that
\begin{equation}
f\xrightarrow{E} h_1 \xrightarrow{E} \dots \xrightarrow{E} h_s \xrightarrow{E} g.
\end{equation}
\end{definition}

We are now ready to state the modified version of the Buchberger criterion.

\begin{theorem}[{\cite[Theorem 3.6]{rutman1992grobner},\cite[Theorem 14]{lezama2008grobner}}]
\label{thm-groe}
For any ideal $I\subset A$, there exists a Groebner basis. Moreover, let $G=\{g_1,\dots,g_m\}$ be a finite subset of non-zero vectors of $A$. We denote by $\mathrm{Syz}(\mathrm{lt}(g_1),\dots,\mathrm{lt}(g_m))$ the submodule of $A^m$ of vectors $(v_1,\dots,v_m)\in A^m$, such that
\begin{equation}
\sum v_i\mathrm{lt}(g_i)=0.
\end{equation}
Let $B$ be a finite generating set of $\mathrm{Syz}(\mathrm{lt}(g_1),\dots,\mathrm{lt}(g_m))$.  Then $G$ is a Groebner basis for $\langle G\rangle$ if and only if for any $(v_1,\dots,v_m)\in B$, we have
\begin{equation}
\sum v_i g_i\xrightarrow{G}_+ 0.
\end{equation}
\end{theorem}

We end this subsection with the following definition.

\begin{definition}
Let $A$ be as above, $I$ an ideal in $A$ and $a_1,\dots,a_l\in A$ non-zero divisors. Then, we define the \textit{saturation of} $I$ \textrm{with respect to} $a_1,\dots,a_l$ by
\begin{equation}
\mathrm{sat}(I,a_1,\dots,a_l)=I_{(S)}\cap A=\{f\in A\mid\textrm{there exists}\,\underline{m}\in\mathbb{N}^l\,\textrm{with}\,\prod a_i^{m_i}f\in I\},
\end{equation}
where $S$ is the multiplicative closure of $a_1,\dots,a_l$ and $I_{(S)}$ the localisation of $I$ with respect to $S$.\\
It is well-known that for formal variable $y_1,\dots,y_l$, we have
\begin{equation}
\mathrm{sat}(I,a_1,\dots,a_l)=\langle I,1-y_1a_1,\dots,1-y_la_l\rangle\cap A.
\end{equation}
Moreover, let $G$ be a Groebner basis of $\langle I,1-y_1a_1,\dots,1-y_la_l\rangle$ with respect to a lexicographic ordering, where $y_1>\dots,y_l>x_1>\dots>x_n$, then $G\cap A$ is a Groebner basis of $\mathrm{sat}(I,a_1,\dots,a_l)$. In particular, we have that $G\cap A$ is a generating set.
\end{definition}

\section{The main conjecture on Mustafin varieties}
In this section, we formulate our main conjecture which predicts a certain type of decomposition of the special fibre for a family of Mustafin varieties. As input data
we fix $d\in\mathbb{Z}_{>0}$, a tuple $\underline{n}=(n_1,\dots,n_{d-1})\in\mathbb{Z}_{>0}^{d-1}$ with $n_1<\dots<n_{d-1}\in\mathbb{Z}_{>0}$, a positive integer $n$ and 
\begin{equation}
\underline{a}\coloneqq\left(a_{ij}^{(l)}\right)_{\substack{i,j=1,\dots,d\\l=0,\dots,n}}\in \mathcal{O}_K^{d^2\cdot (n+1)}.
\end{equation}
For $l=0,\dots,n$, we define
\begin{equation}
M_l=
\begin{pmatrix}
a_{11}^{(l)} & \hdots & a_{1d}^{(l)}\\
\vdots & \ddots & \vdots\\
a_{d1}^{(l)} & \hdots & a_{dd}^{(l)}
\end{pmatrix}
\end{equation}
and
\begin{equation}
\label{equ:matrices}
g_l=M_l\begin{pmatrix}
1\\
& \pi^{n_1}\\
 & & \pi^{n_2}\\
& & & \ddots\\
& & & & \pi^{n_{d-1}} 
\end{pmatrix}
\end{equation}
This yields the lattices $L_l=g_lL$ for $l=0,\dots,n$ and -- denoting $\underline{n}=\{n_1,\dots,n_{d-1}\}$ --  the set $\Gamma_{\underline{a},\underline{n}}\coloneqq\{[L_0],\dots,[L_{n}]\}$.\\
Let $v\in\mathbb{Z}_{\ge0}^{n+1}$, such that $0\le v_i\le d-1$ and $\sum v_i=n(d-1)$. We consider the ideal
\begin{equation}
I_v=\langle \left((x_{ij})_{\substack{j=0,\dots,n}}\right)_{i=1,\dots,v_j}\rangle,
\end{equation}
where if $v_i=0$, we omit the monomials $x_{ij}$.

\begin{definition}
\label{rem:general}
We say that a condition holds for general elements $\underline{a}\in \mathcal{O}_K^{d^2\cdot (n+1)}$, if it holds for all elements in the preimage of  a non-empty Zariski open subset in $\mathbb{A}_{k}^{d^2\cdot (n+1)}$ under the reduction map. In particular, a condition holding for general elements is generically true in $\mathcal{O}_K^{d^2\cdot (n+1)}$.\par
Moreover, let $U\subset \mathbb{A}_K^{d^2\cdot (n+1)}$ be a non-empty Zariski open subset, then, possibly after replacing $K$ by a finite field extension,  $U(K) \cap \mathcal{O}_K^{d^2\cdot (n+1)}$ contains the preimage of a non-empty Zariski open subset in $k^{d^2\cdot (n+1)}$, i.e. it contains a general subset. 
\end{definition}

We are now ready to state our main conjecture.

\begin{conjecture}
\label{conj-must}
Let $\Gamma_{\underline{a},\underline{n}}$ be as above and $\mathcal{M}(\Gamma_{\underline{a},\underline{n}})$ be the corresponding Mustafin variety. Then, there exists a general subset $U_{\underline{n},n}\subset\mathcal{O}_K^{d^2(n+1)}$, such that we have for all $\underline{a}\in U_{\underline{n},n}$ that the ideal of the special fibre decomposes into
\begin{equation}
\label{equ:ideal}
I(\mathcal{M}(\Gamma_{\underline{a},\underline{n}})_k)=\bigcap I_v,
\end{equation}
where the intersection runs over all $v\in\mathbb{Z}_{\ge0}^{n+1}$ with $0\le v_i\le d-1$ and $\sum v_i=n(d-1)$. The primary components correspond to those $v$, such that there exists $i$ with $v_i=0$.
\end{conjecture}

\begin{remark}
\label{rem-step}
We note that it was proved in \cite{conca2007linear} that $\bigcap I_v$ has the same Hilbert polynomial as $I(\mathcal{M}(\Gamma_{\underline{a},\underline{n}})_K)$. As $\mathcal{M}(\Gamma_{\underline{a},\underline{n}})$ is a flat scheme, we have that $I(\mathcal{M}(\Gamma_{\underline{a},\underline{n}})_K)$ and $I(\mathcal{M}(\Gamma_{\underline{a},\underline{n}})_k)$ share the same Hilbert polynomial. Therefore, $\bigcap I_v$ and $I(\mathcal{M}(\Gamma_{\underline{a},\underline{n}})_k)$ also share the same Hilbert polynomial.\\
Let $I,J$ be two multi-homogeneous ideals with $J\subset I$ in a polynomial ring over a field. If $J$ and $I$ share the same Hilbert polynomial, then their radical ideals agree. Since $\bigcap I_v$ and $I(\mathcal{M}(\Gamma_{\underline{a},\underline{n}})_k)$ are radical ideals, it is therefore enough to prove that $\bigcap I_v\subset I(\mathcal{M}(\Gamma_{\underline{a},\underline{n}})_k)$ or $\bigcap I_v\subset I(\mathcal{M}(\Gamma_{\underline{a},\underline{n}})_k)$ in order to derive \cref{conj-must}.
\end{remark}

This conjecture is at the core of most of our considerations for the rest of this paper. For $n_1=1,n_2=2$  and $d=3$ it is proved in \cite[lemma 3.1]{HWplane}, although the proof generalises immediately to arbitrary $1<n_1<n_2$. In \cref{sec:first}, we prove \cref{conj-must} for $d=4$, $2n_1<n_2$, $2n_2<n_3$ and $K=L((\pi))$, where $\pi$ is a formal variable and $L$ an infinite field of characteristic $0$.

\subsection{A proof of \cref{conj-must} for $d=4$ and $2n_1<n_2$ and $2n_2<n_3$}
\label{sec:first}
We now prove \cref{conj-must} for $K=\mathbb{L}(\pi)$,  where $L$ is an infinite field, $d=4$ and $2n_1<n_2$ and $2n_2<n_3$. Our first goal is to show the following proposition.

\begin{proposition}
\label{prop:mustex}
In the setting of \cref{conj-must}, let $d=4$, $\underline{n}$, such that $2n_1<n_2$, $2n_2<n_3$ and $n=3$. Furthermore, we set $K=L((\pi))$, where $\pi$ is a formal variable and $L$ an infinite field with $\mathrm{char}(L)=0$. Then \cref{conj-must} holds.
\end{proposition}

Our proof of \cref{prop:mustex} relies on computations using the computer algebra system \textsc{Singular}.
\begin{proof}
In the setting of \cref{conj-must}, we set $d=4$, $n=3$ and $\underline{n}=(n_1,n_2,n_3)$ with $0<n_1<n_2<n_3$, such that $2n_1<n_2$ and $2n_2<n_3$. Furthermore, we set $K=L((\pi))$, where we set $\pi$ a formal variable. We see that $\mathcal{O}_K=L[[\pi]]$ and $k\cong L$.\\
Then, for given $\underline{a}$, we have have that $I(\mathcal{M}(\Gamma_{\underline{a},\underline{n}})_K)$ is generated by the $2\times2$ minors of

\begin{equation}
\begin{pmatrix}
a_{11}^{(0)}x_{10}+a_{12}^{(0)}\pi x_{20}+a_{13}^{(0)}\pi^3x_{30}+a_{14}^{(0)}\pi^7x_{40} & \hdots
 & a_{11}^{(3)}x_{13}+a_{12}^{(3)}\pi x_{23}+a_{13}^{(3)}\pi^3x_{33}+a_{14}^{(3)}\pi^7x_{43}\\
a_{21}^{(1)}x_{10}+a_{22}^{(0)}\pi x_{20}+a_{23}^{(0)}\pi^3x_{30}+a_{24}^{(0)}\pi^7x_{40} & \hdots &
a_{21}^{(3)}x_{13}+a_{22}^{(3)}\pi x_{23}+a_{23}^{(3)}\pi^3x_{33}+a_{24}^{(3)}\pi^7x_{43} \\
a_{31}^{(0)}x_{10}+a_{32}^{(0)}\pi x_{20}+a_{33}^{(0)}\pi^3x_{30}+a_{34}^{(0)}\pi^7x_{40} & \hdots &
a_{31}^{(3)}x_{13}+a_{32}^{(3)}\pi x_{23}+a_{33}^{(3)}\pi^3x_{33}+a_{34}^{(3)}\pi^7x_{43} \\
a_{41}^{(0)}x_{10}+a_{42}^{(0)}\pi x_{20}+a_{43}^{(0)}\pi^3x_{30}+a_{44}^{(0)}\pi^7x_{40} & \hdots &
a_{41}^{(3)}x_{13}+a_{42}^{(3)}\pi x_{23}+a_{43}^{(3)}\pi^3x_{33}+a_{44}^{(3)}\pi^7x_{43}
\end{pmatrix}
\end{equation}

We denote by $m_{(\alpha,\beta),(\gamma,\delta)}$ the $2\times2$ minor given by the $\alpha-$th and $\beta-$th column and the $\gamma-$th and $\delta-$th row. We observe that
\begin{equation}
I(\mathcal{M}(\Gamma_{\underline{a},\underline{n}}))=\mathrm{sat}\left(\left<m_{(\alpha,\beta),(\gamma,\delta)}\right>_{\mathcal{O}_K}, \pi\right),
\end{equation}
i.e. the ideal of the Mustafin variety over $\mathcal{O}_K$ is given by the saturation of the ideal generated by the $2\times2$ minors with respect to the uniformiser $\pi$. While in principle such a computation could be carried out by a computer algebra system, our computations did not finish. Therefore, we carry out the computation by hand. Our computations were aided by \textsc{Singular} and may be found in \url{https://sites.google.com/site/marvinanashahn/computer-algebra}.\\
A straightforward computation shows that
\begin{align}
\label{equ-exfib}
\begin{split}
\bigcap I_v=&\left<\left(x_{1i}x_{1j}\right)_{\substack{i,j=0,\dots,3\\i<j}},\left(x_{2i}x_{2j}\right)_{\substack{i,j=0,\dots,3\\i<j}},\left(x_{1i}x_{2j}\right)_{\substack{i,j=0,\dots,3\\i\neq j}},\left(x_{1i}x_{3j}\right)_{\substack{i,j=0,\dots,3\\i\neq j}},\right.\\
&\left. \left(x_{2i}x_{3j}x_{3l}\right)_{\substack{i,j,l=0,\dots,3\\i,j,l\,\textrm{p.w. different}}},x_{30}x_{31}x_{32}x_{33}\right>_{k[x_{ij}]},
\end{split}
\end{align}

where the intersection runs over all $v\in\mathbb{Z}_{\ge0}^4$, $0\le v_i\le 4$ and $\sum v_i=n(d-1)=9$. In order to prove that

\begin{equation}
\label{equ-proofst}
I(\mathcal{M}(\Gamma_{\underline{a},\underline{n}})_k)=\bigcap I_v
\end{equation}

we successively combine the minors $m_{(\alpha,\beta),(\gamma,\delta)}$ to find polynomials over $\mathcal{O}_K$ whose reductions yields the generators of $\bigcap I_v$ given in \cref{equ-exfib}.

To begin with, let $\alpha,\beta,\gamma,\delta\in\{0,\dots,3\}$, such that $\alpha\neq\beta,\gamma\neq\delta$. We observe that
\begin{align}
m_{(\alpha,\beta),(\gamma,\delta)}=&\star\pi^{2n_3}x_{4\alpha}x_{4\beta}+\star\pi^{n_2+n_3}x_{3\alpha}x_{4\beta}+\star\pi^{n_2+n_3}x_{4\alpha}x_{3\beta}+\star\pi^{n_1+n_2}x_{2\alpha}x_{4\beta}+\star\pi^{n_1+n_2}x_{4\alpha}x_{2\beta}\\
&+\star\pi^{n_3}x_{1\alpha}x_{4\beta}+\star\pi^{n_3}x_{4\alpha}x_{1\beta}+\star\pi^{2n_2}x_{3\alpha}x_{3\beta}+\star\pi^{n_1+n_2}x_{2\alpha}x_{3\beta}+\star\pi^{n_1+n_2}x_{3\alpha}x_{2\beta}\\
&+\star\pi^{n_2}x_{1\alpha}x_{3\beta}+\star\pi^{n_2}x_{3\alpha}x_{1\beta}+\star\pi^{2n_1}x_{2\alpha}x_{2\beta}+\star\pi^{n_1} x_{1\alpha}x_{2\beta}+\star\pi^{n_1} x_{2\alpha}x_{1\beta}+\star x_{1\alpha}x_{1\beta}
\end{align}
where $\star$ indicate the coefficients, which are non-zero polynomials in the $a_{ij}^{(l)}$, which depend on $\alpha,\beta,\gamma,\delta$ and are independent of $\pi$. Thus, for general coefficients, we have $\star\in \mathcal{O}_K^{\ast}$. We see immediately, that $(m_{(\alpha,\beta),(\gamma,\delta)}\,\mathrm{mod}\, \pi)\in k[x_{ij}]$ generate $\langle (x_{1i}x_{1j})_{\substack{i,j=0,\dots,3\\i<j}}\rangle_{k[x_{ij}]}$.\\
In order to proceed with our computation, we need to introduce some notation. Let $f$ be a polynomial over $\mathcal{O}_K$ and $m$ a monomial in $x_{ij}, \pi$. Then, we denote by $\mathrm{coeff}(f, m)$ the coefficient of $m$ in $f$.\\
As a first step, we define

\begin{align}
m^{(1)}_{1,(\alpha,\beta)}&=\mathrm{coeff}(m_{(\alpha,\beta),(1,3)},x_{1\alpha}x_{1\beta})m_{(\alpha,\beta),(1,2)}-\mathrm{coeff}(m_{(\alpha,\beta),(1,2)},x_{1\alpha}x_{1\beta})m_{(\alpha,\beta),(1,3)}\\
m^{(1)}_{2,(\alpha,\beta)}&=\mathrm{coeff}(m_{(\alpha,\beta),(2,3)},x_{1\alpha}x_{1\beta})m_{(\alpha,\beta),(1,2)}-\mathrm{coeff}(m_{(\alpha,\beta),(1,2)},x_{1\alpha}x_{1\beta})m_{(\alpha,\beta),(2,3)}\\
m^{(1)}_{3,(\alpha,\beta)}&=\mathrm{coeff}(m_{(\alpha,\beta),(1,4)},x_{1\alpha}x_{1\beta})m_{(\alpha,\beta),(1,3)}-\mathrm{coeff}(m_{(\alpha,\beta),(1,3)},x_{1\alpha}x_{1\beta})m_{(\alpha,\beta),(1,4)}\\
m^{(1)}_{4,(\alpha,\beta)}&=\mathrm{coeff}(m_{(\alpha,\beta),(3,4)},x_{1\alpha}x_{1\beta})m_{(\alpha,\beta),(1,3)}-\mathrm{coeff}(m_{(\alpha,\beta),(1,3)},x_{1\alpha}x_{1\beta})m_{(\alpha,\beta),(3,4)}\\
m^{(1)}_{5,(\alpha,\beta)}&=\mathrm{coeff}(m_{(\alpha,\beta),(1,4)},x_{1\alpha}x_{1\beta})m_{(\alpha,\beta),(1,2)}-\mathrm{coeff}(m_{(\alpha,\beta),(1,2)},x_{1\alpha}x_{1\beta})m_{(\alpha,\beta),(1,4)}\\
m^{(1)}_{6,(\alpha,\beta)}&=\mathrm{coeff}(m_{(\alpha,\beta),(2,4)},x_{1\alpha}x_{1\beta})m_{(\alpha,\beta),(1,4)}-\mathrm{coeff}(m_{(\alpha,\beta),(1,4)},x_{1\alpha}x_{1\beta})m_{(\alpha,\beta),(2,4)}.
\end{align}

A \textsc{Singular} computation shows that

\begin{align}
m_{j,(\alpha,\beta)}^{(1)}=&\star\pi^{2n_3}x_{4\alpha}x_{4\beta}+\star\pi^{n_2+n_3}x_{3\alpha}x_{4\beta}+\star\pi^{n_2+n_3}x_{4\alpha}x_{3\beta}+\star\pi^{n_1+n_2}x_{2\alpha}x_{4\beta}+\star\pi^{n_1+n_2}x_{4\alpha}x_{2\beta}\\
&+\star\pi^{n_3}x_{1\alpha}x_{4\beta}+\star\pi^{n_3}x_{4\alpha}x_{1\beta}+\star\pi^{2n_2}x_{3\alpha}x_{3\beta}+\star\pi^{n_1+n_2}x_{2\alpha}x_{3\beta}+\star\pi^{n_1+n_2}x_{3\alpha}x_{2\beta}\\
&+\star\pi^{n_2}x_{1\alpha}x_{3\beta}+\star\pi^{n_2}x_{3\alpha}x_{1\beta}+\star\pi^{2n_1}x_{2\alpha}x_{2\beta}+\star\pi^{n_1} x_{1\alpha}x_{2\beta}+\star\pi^{n_1} x_{2\alpha}x_{1\beta},
\end{align}

where $\star$ indicate non-zero polynomials in the $\underline{a}$, which depend on $j$ and that are independent of $\pi$.  Furthermore, we set

\begin{align}
&m_{1,(\alpha,\beta)}^{(2)}=a_{21}^{(\alpha)}m_1^{(1)}-a_{11}^{(\alpha)}m_2^{(1)},\quad m_{2,(\alpha,\beta)}^{(2)}=a_{31}^{(\alpha)}m_3^{(1)}-a_{11}^{(\alpha}m_4^{(1)}\\
&m_{3,(\alpha,\beta)}^{(2)}=a_{41}^{(\alpha)}m_5^{(1)}-a_{11}^{(\alpha)}m_6^{(1)}.
\end{align}

Again, a Singular computation shows that
\begin{align}
m_{j,(\alpha,\beta)}^{(2)}=&\star\pi^{2n_3}x_{4\alpha}x_{4\beta}+\star\pi^{n_2+n_3}x_{3\alpha}x_{4\beta}+\star\pi^{n_2+n_3}x_{4\alpha}x_{3\beta}+\star\pi^{n_1+n_2}x_{2\alpha}x_{4\beta}+\star\pi^{n_1+n_2}x_{4\alpha}x_{2\beta}\\
&+\star\pi^{n_3}x_{4\alpha}x_{1\beta}+\star\pi^{2n_2}x_{3\alpha}x_{3\beta}+\star\pi^{n_1+n_2}x_{2\alpha}x_{3\beta}+\star\pi^{n_1+n_2}x_{3\alpha}x_{2\beta}+\star\pi^{n_2}x_{3\alpha}x_{1\beta}\\
&+\star\pi^{2n_1}x_{2\alpha}x_{2\beta}+\star\pi^{n_1} x_{2\alpha}x_{1\beta},
\end{align}
where again $\star$ indicate non-zero polynomials in the $\underline{a}$, which depend on $j$ and are independent of $\pi$. Thus, for general coefficients, we have $\star\in \mathcal{O}_K^\ast$. We see that $\pi^{-n_1}m_{l,(\alpha,\beta)}^{(2)}\in \mathcal{O}_K[x_{ij}]$ and that the $\pi^{-n_1}m_{l,(\alpha,\beta)}^{(2)}\,\mathrm{mod}\,\pi\in k[x_{ij}]$ generate $\langle(x_{1i}x_{2j})\rangle_{\substack{i,j=0,\dots,3\\i\neq j}}$.\\
Next, we define
\begin{align}
m_{j,\alpha,\beta}^{(3)}=m_{j+1,\alpha,\beta}^{(2)}-m_{1,\alpha,\beta}^{(2)}
\end{align}
for $j=1,2$. A Singular computation shows that
\begin{align}
m_{j,(\alpha,\beta)}^{(3)}=&\star\pi^{2n_3}x_{4\alpha}x_{4\beta}+\star\pi^{n_2+n_3}x_{3\alpha}x_{4\beta}+\star\pi^{n_2+n_3}x_{4\alpha}x_{3\beta}+\star\pi^{n_1+n_2}x_{2\alpha}x_{4\beta}+\star\pi^{n_1+n_2}x_{4\alpha}x_{2\beta}\\
&+\star\pi^{n_3}x_{4\alpha}x_{1\beta}+\star\pi^{2n_2}x_{3\alpha}x_{3\beta}+\star\pi^{n_1+n_2}x_{2\alpha}x_{3\beta}+\star\pi^{n_1+n_2}x_{3\alpha}x_{2\beta}+\star\pi^{n_2}x_{3\alpha}x_{1\beta}\\
&+\star\pi^{2n_1}x_{2\alpha}x_{2\beta},
\end{align}
where again $\star$ indicate non-zero polynomials in the $\underline{a}$, which depend on $j$ and are independent of $\pi$. Thus, for general coefficients, we have $\star\in \mathcal{O}_K^\ast$. We see that the $\pi^{-2n_1}m_{l,(\alpha,\beta)}^{(3)}\in \mathcal{O}_K[x_{ij}]$ and that the $\pi^{-2n_1}m_{l,(\alpha,\beta)}^{(3)}\,\mathrm{mod}\,\pi\in k[x_{ij}]$ generate $\langle(x_{2i}x_{2j})\rangle_{\substack{i,j=0,\dots,3\\i< j}}$.\\
Now, we define
\begin{align}
m_{\alpha,\beta}^{(4)}=&(-a_{12}^{(\alpha)}a_{21}^{(\alpha)}a_{41}^{(\alpha)}a_{11}^{(\beta)}+a^{(\alpha)}_{11}a^{(\alpha)}_{22}a^{(\alpha)}_{41}a^{(\beta)}_{11}+a^{(\alpha)}_{11}a^{(\alpha)}_{12}a_{21}^{(\alpha)}a^{(\beta)}_{41}-(a^{(\alpha)}_{11})^2a^{(\alpha)}_{22}a^{(\beta)}_{41})m_{1,(\alpha,\beta)}^{(3)}\\
&-(-a^{(\alpha)}_{12}(a^{(\alpha)}_{31})^2a^{(\beta)}_{11}+a^{(\alpha)}_{11}a^{(\alpha)}_{31}a^{(\alpha)}_{32}a^{(\beta)}_{11}+a^{(\alpha)}_{11}a^{(\alpha)}_{12}a^{(\alpha)}_{31}a^{(\beta)}_{31}-(a^{(\alpha)}_{11})^2a^{(\alpha)}_{32}a^{(\beta)}_{31})m_{2,(\alpha,\beta)}^{(3)}.
\end{align}

A Singular computation shows that
\begin{align}
m_{(\alpha,\beta)}^{(4)}=&\star\pi^{2n_3}x_{4\alpha}x_{4\beta}+\star\pi^{n_2+n_3}x_{3\alpha}x_{4\beta}+\star\pi^{n_2+n_3}x_{4\alpha}x_{3\beta}+\star\pi^{n_1+n_2}x_{4\alpha}x_{2\beta}\\
&+\star\pi^{n_3}x_{4\alpha}x_{1\beta}+\star\pi^{2n_2}x_{3\alpha}x_{3\beta}+\star\pi^{n_1+n_2}x_{3\alpha}x_{2\beta}+\star\pi^{n_2}x_{3\alpha}x_{1\beta},
\end{align}

where again $\star$ indicate non-zero polynomials in the $\underline{a}$, which are independent of $\pi$. Thus, for general coefficients, we have $\star\in \mathcal{O}_K^\ast$. We see that the $\pi^{-n_2}m_{(\alpha,\beta)}^{(4)}\,\mathrm{mod}\,\pi\in k[x_{ij}]$ generate $\langle(x_{1i}x_{3j})\rangle_{\substack{i,j=0,\dots,3\\i\neq j}}$.
Now, let $\gamma\in\{0,\dots,3\}\backslash\{\alpha,\beta\}$ and set

\begin{align}
m_{(\alpha,\beta,\gamma)}&=\mathrm{coeff}(m_{\alpha,\beta}^{(4)},\pi^3x_{1\beta}x_{3\alpha})x_{3\alpha}m_{(\gamma,\beta)}^{(4)}-\mathrm{coeff}(m_{\gamma,\beta}^{(4)},\pi^3x_{1\beta}x_{3\gamma})x_{3\gamma}m_{(\alpha,\beta)}^{(4)}.
\end{align}

We compute in Singular, that

\begin{align}
m_{(\alpha,\beta,\gamma)}=&\star\pi^{2n_3}x_{4\alpha}x_{4\beta}x_{3\gamma}
+\star\pi^{2n_3}x_{3\alpha}x_{4\beta}x_{4\gamma}
+\star\pi^{n_2+n_3}x_{3\alpha}x_{4\beta}x_{3\gamma}
+\star\pi^{n_2+n_3}x_{3\alpha}x_{4\beta}x_{3\gamma}\\
&+\star\pi^{n_2+n_3}x_{4\alpha}x_{3\beta}x_{3\gamma}+\star\pi^{n_2+n_3}x_{3\alpha}x_{3\beta}x_{4\gamma}
+\star\pi^{n_1+n_2}x_{4\alpha}x_{2\beta}x_{3\gamma}
+\star\pi^{n_1+n_2}x_{3\alpha}x_{2\beta}x_{4\gamma}\\
&+\star\pi^{n_3}x_{4\alpha}x_{1\beta}x_{3\gamma}
+\star\pi^{n_3}x_{3\alpha}x_{1\beta}x_{4\gamma}
+\star\pi^{2n_2}x_{3\alpha}x_{3\beta}x_{3\gamma}\\
&+\star\pi^{n_1+n_2}x_{3\alpha}x_{2\beta}x_{3\gamma},
\end{align}
where again $\star$ indicate non-zero polynomials in the $\underline{a}$, which are independent of $\pi$. Thus, for general coefficients, we have $\star\in \mathcal{O}_K^\ast$. We see that $\pi^{-n_1-n_2}m_{(\alpha,\beta,\gamma)}\in \mathcal{O}_K[x_{ij}]$ and that
the $m_{(\alpha,\beta,\gamma)}\,\mathrm{mod}\,\pi\in k[x_{ij}]$ generate $\langle(x_{2i}x_{3j}x_{3l})\rangle_{\substack{i,j,l=0,\dots,3\\i,j,l\,\textrm{p.w. different}}}$.\\
Finally, let $\{\alpha,\beta,\gamma,\delta\}=\{0,\dots,3\}$ and define

\begin{align}
m_{(\alpha,\beta,\gamma,\delta)}=\mathrm{coeff}(m_{\alpha,\beta,\gamma},\pi^{n_1+n_2}x_{3\alpha}x_{2\beta}x_{3\gamma})x_{3\gamma}m_{\alpha,\beta,\delta}-\mathrm{coeff}(m_{\alpha,\beta,\delta},\pi^{n_1+n_2}x_{3\alpha}x_{2\beta}x_{3\delta})x_{3\delta}m_{\alpha,\beta,\gamma}.
\end{align}

We see that

\begin{align}
m_{(\alpha,\beta,\gamma,\delta)}=&\star\pi^{2n_3}x_{4\alpha}x_{4\beta}x_{3\gamma}x_{3\delta}
+\star\pi^{2n_3}x_{3\alpha}x_{4\beta}x_{4\gamma}x_{3\delta}
+\star\pi^{2n_3}x_{3\alpha}x_{4\beta}x_{3\gamma}x_{4\delta}\\
&+\star\pi^{n_2+n_3}x_{3\alpha}x_{4\beta}x_{3\gamma}x_{3\delta}
+\star\pi^{n_2+n_3}x_{4\alpha}x_{3\beta}x_{3\gamma}x_{3\delta}
+\star\pi^{n_2+n_3}x_{3\alpha}x_{3\beta}x_{4\gamma}x_{3\delta}\\
&+\star\pi^{n_2+n_3}x_{3\alpha}x_{3\beta}x_{3\gamma}x_{4\delta}
+\star\pi^{n_1+n_2}x_{4\alpha}x_{2\beta}x_{3\gamma}x_{3\delta}
+\star\pi^{n_1+n_2}x_{3\alpha}x_{2\beta}x_{4\gamma}x_{3\delta}\\
&+\star\pi^{n_1+n_2}x_{3\alpha}x_{2\beta}x_{3\gamma}x_{4\delta}
+\star\pi^{n_3}x_{4\alpha}x_{1\beta}x_{3\gamma}x_{3\delta}
+\star\pi^{n_3}x_{3\alpha}x_{1\beta}x_{4\gamma}x_{3\delta}\\
&+\star\pi^{n_3}x_{3\alpha}x_{1\beta}x_{3\gamma}x_{4\delta}
+\star\pi^{2n_2}x_{3\alpha}x_{3\beta}x_{3\gamma}x_{3\delta},
\end{align}
where $\star$ indicate polynomials in the $\underline{a}$, which are independent of $\pi$. Unfortunately, our Singular computation of $m_{(\alpha,\beta,\gamma,\delta)}$ did not finish for $\underline{a}$ as variable parameters. However, for random choices of $\underline{a}$ the term $\mathrm{coeff}(m_{(\alpha,\beta,\gamma,\delta)},\pi^{2n_2}x_{3\alpha}x_{3\beta}x_{3\gamma}x_{3\delta})$ does not vanish. Therefore, $\mathrm{coeff}(m_{(\alpha,\beta,\gamma,\delta)},\pi^{2n_2}x_{3\alpha}x_{3\beta}x_{3\gamma}x_{3\delta})$ is a non-zero polynomial in the $\underline{a}$, which is independent of $\pi$. We observe that $\pi^{-2n_2}m_{(\alpha,\beta,\gamma,\delta)}\in \mathcal{O}_K[x_{ij}]$ and therefore $\pi^{-2n_2}m_{(\alpha,\beta,\gamma,\delta)}\,\mathrm{mod}\,\pi\in k[x_{ij}]$ generates $\langle x_{30}x_{31}x_{32}x_{33}\rangle$.\vspace{\baselineskip}

This computation shows that in the notation of \cref{equ-proofst} that
\begin{equation}
\bigcap I_v\subset I(\mathcal{M}(\Gamma_{\underline{a},\underline{n}})_k).
\end{equation}
By \cref{rem-step}, the proposition follows.
\end{proof}

We immediately obtain the following corollary.

\begin{corollary}
\label{cor:mustex}
In the setting of \cref{conj-must}, let $d=4$ and $\underline{n}$, such that $2n_1<n_2$, $2n_2<n_3$. Furthermore, we set $K=L((\pi))$ with $\mathrm{char}(L)=0$, where $\pi$ is a formal variable. Then \cref{conj-must} holds.\\
Moreover, for the same parameters $\underline{n},d$, we have \cref{conj-must} holds for $K=\mathbb{Q}_p$ and $p\gg0$.
\end{corollary}

\begin{proof}
The corollary first states the generalisation of \cref{prop:mustex} to $n>3$. Its proof is completely analogous to \cite[lemma 3.2]{HWplane}. The statement regarding $K=\mathbb{Q}_p$ follows by standard considerations.
\end{proof}

\begin{remark}
The \textsc{Singular} computations in the proof of \cref{prop:mustex} give explicit formulas for all involved polynomials except for $m_{(\alpha,\beta,\gamma,\delta)}$ for which the computations did not finish. A close examination of the polynomials shows that the relevant polynomials in the $\underline{a}$ do not completely vanish in any characteristic. It is reasonable to expect the same for the coefficient of $\pi^{2n_2}x_{3\alpha}x_{3\beta}x_{3\gamma}x_{3\delta}$ in $m_{(\alpha,\beta,\gamma,\delta)}$, which would give a proof of \cref{conj-must} in this setting for any field $K$ with infinite residue field $k$.
\end{remark}

To end this section, we outline a possible strategy involving a Groebner basis approach towards \cref{conj-must}.

\begin{remark}
\label{rem-genin}
One interesting observation regarding the proposed equality in \cref{equ:ideal} is that the right hand side is the unique \textit{multigraded Borel-fixed ideal in the Hilbert scheme} of the diagonal embedding $\mathbb{P}(V)\hookrightarrow\mathbb{P}(V)^{(n+1)}$ \cite[theorem 2.1]{cartwright2010hilbert}.\\
To make this precise, we let $\mathrm{GL}(V)$ act on $k[x_1,\dots,x_d]$  via $g\cdot (x_1,\dots,x_d)\coloneqq (x_1,\dots,x_d)g$, i.e. by multiplication from the right. When computing Groebner basis $G$ of an ideal $I$ with respect to a term order $<$, the \textit{initial ideal} $\mathrm{Lt}_<(I)$ depends on the coordinates $x_1,\dots,x_d$, we chose. In particular, $\mathrm{Lt}_<(I)$ is not invariant under the action of $\mathrm{GL}(V)$ on $I$. However, for any term order $<$, there exists a Zariski open subset $U\subset\mathrm{GL}(V)$, such thatf for a homogeneous ideal $I$ we have that $\mathrm{Lt}(g\cdot I)$ is constant for all $g\in U$ \cite[theorem 15.20]{eisenbud2013commutative}. We call $gen(I)=\mathrm{Lt}(g\cdot I)$ the \textit{generic initial ideal} of $I$. Generic initial ideals ideals play an important role in algebraic geometry. One key result is a theorem by Galligo (in characteristic $0$) and Bayer-Stillman (in positive characteristic) \cite{galligo1979theoreme,bayer1987theorem}, which states that the generic initial ideal is \textit{Borel-fixed}. This means that for the \textit{Borel subgroup} $B\subset \mathrm{GL}(V)$ of upper triangular matrices, we have $b\cdot gen(I)I=gen(I)$ for all $b\in B$.  This may be generalised to the case of multigraded ideals. More precisely, let $I\subset K[(x_i^{(l)})_{\substack{l=0,\dots,n\\i=1,\dots,\alpha_l}}]$ be multigraded, i.e. homogeneous in the $(x_{i}^{(l)})$ for each fixed $l$ and $<$ a term order and let $\mathrm{GL}(V)^{n+1}$ act on $K[(x_i^{(l)})_{\substack{l=0,\dots,n\\i=1,\dots,d}}]$ via $(g_0,\dots,g_n)\cdot ((x_i^{(l)})_{i,l})=((x_1^{(l)},\dots,x_{d}^{(l)})g_l)_{l})$. Then, there exists a Zariski open subset $U\subset \mathrm{GL}(V)^{n+1}$, such that $\mathrm{Lt}_<(g\cdot I)$ is constant for $g\in U$. We call $mgen(I)=\mathrm{Lt}_<(g\cdot I)$ the \textit{multigraded generic initial ideal} of $I$. Moreover, the ideal $mgen(I)$ is Borel-fixed in the multigraded sense, i.e. $b\cdot mgen(I)=mgen(I)$ for all $b\in B^{n+1}\subset\mathrm{GL}(V)$.\\
In \cite[section 6]{cartwright2010hilbert}, it is suggested that the ideal on the left hand side of \cref{equ:ideal} should be the initial ideal of an appropriate linear twist of the aforementioned diagonal embedding with respect to a term order induced by $\underline{n}$. Therefore, a possible strategy towards a general proof of \cref{conj-must} for $K=\mathbb{Q}(\pi)$ is to show that $I(\mathcal{M}(\Gamma_{\underline{a},\underline{n}})_k)$ is constant and Borel-fixed in the multigraded sense for general $\underline{a}$.
\end{remark}

\section{Mustafin models of projective varieties and syzygy bundles}

We begin this section with the following definition of models of projective varieties induced by the construction of Mustafin varieties.

\begin{definition}
Let $\Gamma=\{[L_1],\dots,[L_n]\}$ be a finite set of homothety classes of lattices and $\mathcal{M}(\Gamma)$ be the associated Mustafin variety. Furthermore, let $X\subset\mathbb{P}(V)$. Considering the map \cref{equ:mapmust}, we define $\overline{f_{\Gamma}(X)}$ endowed with the reduced scheme structure of $X$ the \textit{Mustafin model of $X$ associated to $\Gamma$}, which we denote by $X(\Gamma_{\underline{a},\underline{n}})$.
\end{definition}

We note that $X(\Gamma_{\underline{a},\underline{n}})$ is a flat and proper scheme with generic fibre $X$. It is an interesting open problem to classify what kinds of models of $X$ one may obtain with this construction. A first step towards this classification was made by the authors in \cite{HWplane}.

\begin{theorem}[{\cite[Theorem 3.4]{HWplane}}]
\label{thm-old}
Let $n_1=1,n_2=2$ and $d=3$. Assume that the residue field $k$ of $K$ is perfect and let $n\ge 2$. Let $X\subset\mathbb{P}(V)$ be an irreducible plane curve. Then, possibly after a finite unramified field extension, for general $\underline{a}$ we have that
\begin{enumerate}
\item $X(\Gamma_{\underline{a},\underline{n}})_k$ decomposes into $n+1$ irreducible components $X_0,\dots,X_{n}$;
\item $X_l$ is contained in the $l-$th primary component of $\mathcal{M}(\Gamma_{\underline{a},\underline{n}})$, $X_l^{\mathrm{red}}\cong\mathbb{P}^1$ and all components intersect pairwise in the same point. More precisely, $X_i\subset\left(\mathbb{P}_k^2\right)^{n+1}$ is cut out by
\begin{equation}
\langle x_{1l},(x_{1i},x_{2i})_{i\neq l}\rangle.
\end{equation}
\end{enumerate}
\end{theorem}

One of the two main ingredients for this proof is \cref{conj-must} for $n_1=1,n_2=2$ and $d=3$ (\cite[lemma 3.1]{HWplane}).\\
Our main goal in this section is the following theorem, which generalises \cite[theorem 3.4]{HWplane}. We will give the proof in \cref{sec:proof}

\begin{theorem}
\label{thm-main1}
We assume that \cref{conj-must} holds. Furthermore, assume that $k$ is a perfect field. Let $X\subset\mathbb{P}(V)$ be a geometrically irreducible projective variety of dimension $d'$, such that $\mathrm{dim}(X)=\mathrm{dim}(X_{\overline{K}})$, where $X_{\overline{K}}$ denote the pullback of $X$ to the algebraic closure of $K$, $\Gamma_{\underline{a},\underline{n}}=\{[L_0],\dots,[L_n]\}$ be a set of lattices as in \cref{conj-must} and let $X(\Gamma_{\underline{a},\underline{n}})$ be the associated Mustafin degeneration of $X$. Then, possibly after a finite unramified field extension of $K$, there exists a general subset $U\subset \mathcal{O}_K^{d^2(n+1)}$, such that for all $\underline{a}\in U$ we have that the special fibre $X(\Gamma_{\underline{a},\underline{n}})_k$ is contained in $\mathcal{M}(\Gamma_{\underline{a},\underline{n}})_{k,\le d'}$.
\end{theorem}

\begin{remark}
We note that the statement in \cite[theorem 3.4]{HWplane} also gives ideals of the irreducible components of $X(\Gamma_{\underline{a},\underline{n}})_k$ for smaller general set of $\underline{a}$. Interestingly, these ideals  are Borel-fixed in the multigraded sense. It is thus tempting to ask whether the defining ideal of $X(\Gamma_{\underline{a},\underline{n}})_k$ is always Borel-fixed in this sense for sufficiently general choices of $\underline{a}$.
\end{remark}

The following is an immediate corollary of \cref{thm-main1}.

\begin{corollary}
\label{cor-curves}
We assume that \cref{conj-must} holds. Furthermore, assume that $k$ is a perfect field. Let $X\subset \mathbb{P}(V)$ be a projective, geometrically irreducible curve, , such that $\mathrm{dim}(X)=\mathrm{dim}(X_{\overline{K}})$, $\Gamma_{\underline{a},\underline{n}}=\{[L_0],\dots,[L_n]\}$ a set of lattices as in \cref{conj-must} and let $X(\Gamma_{\underline{a},\underline{n}})$ be the associated Mustafin degeneration of $X$. Then, for sufficiently general $\underline{a}$,  any irreducible component of $X(\Gamma_{\underline{a},\underline{n}})_k$ is contained in a primary component of $\mathcal{M}(\Gamma_{\underline{a},\underline{n}})$.
\end{corollary}

This corollary motivates the following definition.

\begin{definition}
Let $X\subset \mathbb{P}(V)$ be a projective curve, $\Gamma_{\underline{a},\underline{n}}=\{[L_0],\dots,[L_n]\}$ a set of lattices as in \cref{conj-must} and $X(\Gamma_{\underline{a},\underline{n}})$ the associated Mustafin degeneration of $X$. We say $X(\Gamma_{\underline{a},\underline{n}})$ has \textit{star-like reduction} if $X(\Gamma_{\underline{a},\underline{n}})_k\subset \bigcup V(I_v)$,
where the union runs over all $v\in\mathbb{Z}^{n+1}$, such that $\sum v_i=n(d-1)$, $0\le v_i\le d-1$ and there exists $i$, such that $v_i=d-1$.
\end{definition}

As a first step towards \cref{thm-main1}, we prove the following lemma.

\begin{lemma}
\label{lem-upbound}
Let $X\subset \mathbb{P}(V)$ be an irreducible projective variety of degree $d'$. Then, $X(\Gamma_{\underline{a},\underline{n}})_k$ has at most
\begin{equation}
d'\binom{n+\mathrm{dim}(X)}{\mathrm{dim}(X)-1}
\end{equation}
irreducible components.
\end{lemma}

\begin{proof}
Recall that the Chow ring of $\mathbb{P}(V)^{n+1}$ is given by
\begin{equation}
\mathcal{A}=\faktor{\mathbb{Z}[H_0,\dots,H_n]}{\langle H_0^d,\dots,H_n^d\rangle},
\end{equation}
where $H_i$ is the hyperplane class of the $i-$th factor. The Chow-class of $X$ in the Chow ring $\faktor{\mathbb{Z}}{\langle H^d\rangle}$ of $\mathbb{P}(V)$ is given by $d'H^{\mathrm{codim}(X)}$, where $H$ is the hyperplane class in $\mathbb{P}(V)$. Under the diagonal embedding $\Delta:\mathbb{P}(V)\to\mathbb{P}(V)^{n+1}$, the class $H^{d-l}$ pushes forward to
\begin{equation}
\sum_{\substack{\underline{m}=(m_0,\dots,m_n)\\\sum m_i=(n+1)(d-1)-l}}\prod_{i=0}^nH_i^{m_i}.
\end{equation}
Thus, the Chow class of $\Delta(X)$ is given by 
\begin{equation}
\label{equ:chow}
d'\cdot\sum_{\substack{\underline{m}=(m_0,\dots,m_n)\\\sum m_i=(n+1)(d-1)-\mathrm{dim}(X)}}\prod_{i=0}^nH_i^{m_i}.
\end{equation}
This is also the Chow class of $X(\Gamma_{\underline{a},\underline{n}})_K$, as it is a linear transformation of $\Delta(X)$. As $X(\Gamma_{\underline{a},\underline{n}})_k$ is the specialisation of $X(\Gamma_{\underline{a},\underline{n}})_K$, their Chow classes agree (see the discussion prior to \cite[corollary 20.3]{fulton2013intersection}). The Chow class of $X(\Gamma_{\underline{a},\underline{n}})_k$ is the sum of the Chow classes of its irreducible components. As all these classes are effective, we see that we may have at most $d'$ times the number of monomials in \cref{equ:chow}. The number of those monomials is easily seen to be
\begin{equation}
\binom{n+\mathrm{dim}(X)}{\mathrm{dim}(X)-1},
\end{equation}
which completes the proof.
\end{proof}

Before giving the proof of \cref{thm-main1} in \cref{sec:proof}, we first prove some technical lemmata and introduce a suitable formal set-up

\subsection{Technical lemmata}
\label{sec-tech}
We begin with a general consideration of how Groebner bases of ideals behave after substitution. Let $\cA$ be a noetherian domain and $A_1,\dots,A_m,x_1,\dots,x_n$ be formal variables. We further define
\begin{equation}
\mathfrak{A}=\cA[A_1,\dots,A_m].
\end{equation}

Let $\mathfrak{I}\subset\mathfrak{\cA}[x_1,\dots,x_n]$ be an ideal and let $\mathfrak{G}$ be a Groebner basis of $\mathfrak{I}$. We further pick $\underline{a}=(a_1,\dots,a_m)\in \cA^m$ and consider the homomorphism
\begin{equation}
\mathrm{Subst}_{\underline{a}}:\mathfrak{A}[x_1,\dots,x_n]\to \cA[x_1,\dots,x_n]
\end{equation}
induced by $A_i\mapsto a_i$. It is an interesting question for which choices of $\underline{a}$, we have that $\mathrm{Subst}_{\underline{a}}(\fG)$ is a Groebner basis of $\mathrm{Subst}_{\underline{a}}(\fI)$. When $\cA$ is a field this is a well-studied topic and has lead to the notion of comprehensive Gröbner bases \cite{weispfenning1992comprehensive,weispfenning2003canonical}.

The following lemma gives a sufficient criterion when $\cA$ is a unique factorisation domain.

\begin{lemma}
\label{lem-critgroeb}
Let $\cA$ be a unique factorisation domain and $\mathfrak{I}$ as above. Let $G=\{f_1,\dots,f_l\}$ be a Groebner basis of $\fI$ with respect to some order $<$. Then, there exist two finite sets of non-zero polynomials $(\mathfrak{P}_i)_{i=1,\dots,l},(\tilde{\mathfrak{P}}_j)_{j=1,\dots,\mu}\in\mathfrak{A}$ for some $\mu>0$, such that the following holds: 
\begin{enumerate}
\item If $\alpha\in \cA$, such that $\alpha$ is a factor of any $\mathfrak{P}_i$, then $\alpha$ is a unit. In other words, each $\mathfrak{P}_i$ is saturated with respect to any non-unit in $\cA$,
\item for all $\underline{a}\in \cA^m$ for which $\mathrm{Subst}_{\underline{a}}(\mathfrak{P}_i)$ is a unit for all $i=1,\dots,l$ and $\mathrm{Subst}_{\underline{a}}(\tilde{\mathfrak{P}}_j)\neq0$ for all $j=1,\dots,\mu$, we have $\mathrm{Subst}_{\underline{a}}(G)$ is a Groebner basis of $\mathrm{Subst}_{\underline{a}}(\mathfrak{I})$m
\item for all $i$, we have $\mathrm{Subst}_{\underline{a}}(\mathrm{lt}(f_i)=\mathrm{lt}(\mathrm{Subst}_{\underline{a}}(f_i))$.
\end{enumerate}
\end{lemma}

\begin{proof}
As $\cA$ is a unique factorisation domain, the ring $\cA[A_1,\dots,A_m]$ is a unique factorisation domain as well. We therefore may consider the following factorisations into irreducible elements
\begin{equation}
\mathrm{lt}(f_i)=\prod_{j=1}^{s_i}p_j^{(i)}.
\end{equation}
We may assume that for suitable $t_i$, we have $p_{1},\dots,p_{t_i}\notin \cA$ and $p_{t_i+1},\dots,p_{s_i}\in \cA$. We then set $\mathfrak{P}_i=\prod_{j=1}^{t_i}p_j^{(i)}$ for $i=1,\dots,l$. Thus, the polynomial $\mathfrak{P}_i$ is saturated with respect to any non-unit of $\cA$.\\
Let $S$ be the multiplicative closure of $(\mathfrak{P}_i)_{i=1,\dots,l}$ and consider the localisation $\mathfrak{A}[(x_i)]_{(S)}$. Then, it follows immediatley from \cref{def-groe} that for $\tilde{f}_i=\frac{f_i}{\mathfrak{P}_i}$, we have that
\begin{equation}
\tilde{G}=\{\tilde{f_1},\dots,\tilde{f}_m\}
\end{equation}
is a Groebner basis of $\mathfrak{I}_{(S)}\subset\mathfrak{A}[(x_i)_i]_{(S)}$. We now consider the homomorphism
\begin{equation}
\mathrm{Subst}_{\underline{a},(S)}:\mathfrak{U}[x_1,\dots,x_n]_{(S)}\to\mathcal{A}[x_1,\dots,x_n]
\end{equation}
induced by $A_i\mapsto a_i$. This is well-defined whenever $\mathrm{Subst}_{\underline{a}}(\mathfrak{P}_i)$ is a unit for all $i=1,\dots,l$. In particular, we then have
\begin{equation}
\left(\mathrm{Subst}_{\underline{a},(S)}\right)_{|\mathfrak{A}[x_1,\dots,x_n]}\equiv\mathrm{Subst}_{\underline{a}}
\end{equation}
and $\mathrm{Subst}_{\underline{a},(S)}(\mathfrak{I}_(S))=\mathrm{Subst}_{\underline{a}}(\mathfrak{I})$. Since $\mathrm{lt}(\tilde{f}_i)\in\mathcal{A}[x_1,\dots,x_n]$, we observe that
\begin{equation}
\mathrm{Syz}(\mathrm{lt}(\mathrm{Subst}_{\underline{a},(S)}(\tilde{f}_1)),\dots,\mathrm{lt}(\mathrm{Subst}_{\underline{a},(S)}(\tilde{f}_n))
\subset \mathrm{Syz}(\mathrm{lt}(\tilde{f_1}),\dots,\mathrm{lt}(\tilde{f}_m)).
\end{equation}
Let $B=(\underline{b}^{(1)},\dots,\underline{b}^{(\alpha)})$ be a finite generating set of $\mathrm{Syz}(\mathrm{lt}(\tilde{f}_1,\dots,\mathrm{lt}(\tilde{f}_n))$. As $\tilde{G}$ is Groebner basis of $\mathfrak{I}_{(S)}$, by \cref{thm-groe} there exists for any $\underline{b}^{p}=(b_1^{(p)},\dots,b_n^{(p)})\in B$ with $p\in\{1,\dots,p\}$ a set of polynomials $h_1^{(p)},\dots,h_{r_p}^{(p)}$, such that
\begin{equation}
\sum_{j=1}^{b_n^{(p)}} b_j^{(p)}\tilde{f}_j\xrightarrow{\tilde{G}}h_1^{(p)}\xrightarrow{\tilde{G}} \dots\xrightarrow{\tilde{G}} h_{r_p}^{(p)}\xrightarrow{E} 0. 
\end{equation}
If $\underline{a}$ is such that $\mathrm{Subst}_{\underline{a},(S)}\mathrm{lt}(h_j^{(p)})\neq0$ for all $j$ and $p$, this immediately implies
\begin{align}
\sum_{j=1}^{b_n^{(p)}} b_j^{(p)}\mathrm{Subst}_{\underline{a},(S)}(\tilde{f}_j)&\xrightarrow{\mathrm{Subst}_{\underline{a},(S)}(\tilde{G})}\mathrm{Subst}_{\underline{a},(S)}(h_1^{(p)})\xrightarrow{\mathrm{Subst}_{\underline{a},(S)}(\tilde{G})} \dots\\
\dots &\xrightarrow{\mathrm{Subst}_{\underline{a},(S)}(\tilde{G})} \mathrm{Subst}_{\underline{a},(S)}((h_{r_p}^{(p)})\xrightarrow{\mathrm{Subst}_{\underline{a},(S)}(\tilde{G})} 0. 
\end{align}
Thus, it follows from \cref{thm-groe} that when $\tilde{G}$ is a Groebner basis of $\mathfrak{I}_{(S)}$, then $\mathrm{Subst}_{\underline{a},(S)}(\tilde{G})$ is a Groebner basis of $\mathrm{Subst}_{\underline{a},S}(\mathfrak{I})=\mathrm{Subst}_{\underline{a}}(\mathfrak{I})$. Finally, we see immediately that -- since $\mathrm{Subst}_{\underline{a},(S)}(\tilde{G})$ and $\mathrm{Subst}_{\underline{a}}(\mathfrak{I})$ only differ by multiplication of units -- $\mathrm{Subst}_{\underline{a},(S)}(\tilde{G})$ is a Groebner basis  if and only if $\mathrm{Subst}_{\underline{a}}(G)$ is as well. Setting $(\tilde{\mathfrak{P}}_j)_{j=1,\dots,\mu}$ the nominators of $(\mathrm{lt}(h_j^{(p)}))_{\substack{p=1,\dots,\alpha\\j=1,\dots,r_p}}$, this finishes the proof.
\end{proof}

The following is an immediate corollary of \cref{lem-critgroeb} and vital for our proof of \cref{conj-must}. We define $\mathfrak{I}=\mathrm{sat}(\tilde{\mathfrak{I}},a)$, i.e. the saturation of $\tilde{\mathfrak{I}}$ with respect to $a$. It is well-known that for a formal variable $y$, we have
\begin{equation}
\mathfrak{I}=\left\langle\tilde{\mathfrak{I}},1-ay\right\rangle_{\mathfrak{A}[(x_{i})_i,y]}\cap\mathfrak{A}[(x_{i})_i].
\end{equation}

We have the following corollary.

\begin{corollary}
\label{lem-redcrit}
Let $\cA$ be a unique factorisation domain and $\tilde{\mathfrak{I}}\subset\mathfrak{\cA}[x_1,\dots,x_n]$ an ideal. Further, let $a\in \mathfrak{U}$ and $G=\{f_1,\dots,f_l\}$ be a Groebner basis of $\mathfrak{I}=(\tilde{\mathfrak{I}},1-ay)$ with respect to some lexicographic order $<_{lex}$, where $y$ is the biggest element in $\{(x_i)_i,y\}$. It is well-known that $G\cap \mathfrak{A}[x_1,\dots,x_n]$ is a Groebner basis of $\mathrm{sat}(\tilde{\mathfrak{I}},a)$.\\
Furthermore, there exist two finite sets of non-zero polynomials $(\mathfrak{P}_i)_{i=1,\dots,l},(\tilde{\mathfrak{P}}_j)_{j=1,\dots,\mu}\in\mathfrak{A}$ for some $\mu>0$, such that the following holds: 
\begin{enumerate}
\item If $\alpha\in \cA$, such that $\alpha$ is a factor of any $\mathfrak{P}_i$, then $\alpha$ is a unit. In other words, each $\mathfrak{P}_i$ is saturated with respect to any non-unit in $\cA$,
\item for all $\underline{a}\in \cA^m$ for which $\mathrm{Subst}_{\underline{a}}(\mathfrak{P}_i)$ is a unit for all $i=1,\dots,l$ and $\mathrm{Subst}_{\underline{a}}(\tilde{\mathfrak{P}}_j)\neq0$ for all $j=1,\dots,\mu$, we have $\mathrm{Subst}_{\underline{a}}(G)\cap \cA[x_1,\dots,x_n]$ is a Grobner basis of $\mathrm{sat}(\mathrm{Subst}_{\underline{a}}(\tilde{\mathfrak{I}}))$.
\end{enumerate}
In particular, we have
\begin{equation}
\label{equ-ideq}
\mathrm{Subst}_{\underline{a}}(\mathrm{sat}(\mathfrak{\tilde{I}},\mathrm{Subst}_{\underline{a}}(a)))=\mathrm{Subst}_{\underline{a}}(\mathfrak{I})\cap\cA[x_1,\dots,x_n]=\mathrm{sat}(\mathrm{Subst}(\mathfrak{\tilde{I}}),\mathrm{Subst}_{\underline{a}}(a)).
\end{equation}
We note that under those conditions, we have that $\mathrm{Subst}_{\underline{a}}(\mathrm{sat}(\mathfrak{\tilde{I}},\mathrm{Subst}_{\underline{a}}(a)))$ is saturated with respect to $a$.
\end{corollary}

\begin{proof}
Let $G$ be as stated in the lemma. It is well-known that $\tilde{G}=G\cap\mathfrak{A}[(x_{i})]$ is a Groebner basis of $\mathfrak{I}=\mathrm{sat}(\tilde{\mathfrak{I}},a)$ with respect to the lexicographic order $<_{lex}$ restricted to $\mathfrak{A}[(x_{i})]$. By \cref{lem-critgroeb}, there exist polynomials $\mathfrak{P}_i,\tilde{\mathfrak{P}}_i$ as stated in the corollary, such that we have $\mathrm{Subst}_{\underline{a}}(G)$ is a Groebner basis of $\mathrm{Subst}_{\underline{a}}(\mathfrak{I})=\langle\mathrm{Subst}_{\underline{a}}(\tilde{\mathfrak{I}}),1-\mathrm{Subst}_{\underline{a}}(a)\rangle$ and $\mathrm{Subst}_{\underline{a}}(G)\cap \cA[(x_i)_i]$ is a Groebner basis of $\mathrm{sat}(\mathrm{Subst}_{\underline{a}}(\mathfrak{I}),\mathrm{Subst}_{\underline{a}}(a))$ whenever $\mathrm{Subst}_{\underline{a}}(\mathfrak{P}_i)$ is a unit and $\mathrm{Subst}_{\underline{a}}(\tilde{\mathfrak{P}}_i)\neq0$. Finally, we observe that by the third assertion in \cref{lem-critgroeb} we have $\mathrm{Subst}_{\underline{a}}(\tilde{G})=\mathrm{Subst}_{\underline{a}}(G)\cap \cA[(x_i)_i]$ and the corollary follows.
\end{proof}

We illustrate \cref{lem-redcrit} in the following example.

\begin{example}
Let $\cA=\mathbb{Z}_p$ the $p-$adic ring of integers and consider the ideal $I\subset\mathbb{Z}[A_1,A_2][x,y]$ given by $I=\langle pA_1x+A_2y\rangle$. For $a_1,a_2\in\mathbb{Z}_p$, we see that $\tilde{I}=\langle pa_1x+a_2y\rangle$ is saturated with respect to $\pi$ if and only if $a_2\in\mathbb{Z}_p^{\times}$. In order to see this in the flavour of \cref{lem-redcrit}, we consider the ideal $J=\langle pA_1x+A_2y,1-pz\rangle\subset\mathbb{Z}_p[A_1,A_2][x,y,z]$ and observe that
\begin{align}
&z\cdot\left(pA_1x+A_2y\right)+A_1x\cdot\left(1-pz\right)=A_1x+A_2yz\in J\\
\end{align}
Therefore $A_2yz\in \mathrm{Lt}_<(J)$ and it is easy to see that no proper factor of $A_2yz$ is contained in $\mathrm{Lt}_<(J)$. Therefore, any Gröbner basis of $J$ must contain a polynomial with leading term $A_2yz$. Thus, it follows that $\mathrm{Subst}_{\underline{a}}(A_2)$ must be invertible in order for \cref{lem-redcrit} to apply.
\end{example}

\subsection{Formal set-up}
We begin with a general set-up of ideals mirroring the ideal in \cref{equ:ideal} over arbitrary noetherian domains.\par
We first introduce our algebraic set-up: Let $\cA$ be a unique factorisation domain and $\left(A_{ij}^{(l)}\right)_{\substack{i,j=1,\dots,d\\l=0,\dots,n}}$ be formal variables. We set the polynomial ring 
\begin{equation}
\mathfrak{A}_{\cA}=\cA\left[\left(A_{ij}^{(l)}\right)_{\substack{i,j=1,\dots,d\\l=0,\dots,n}}\right]
\end{equation}
and for $a\in\mathfrak{A}$, we set the matrices
\begin{equation}
\mathfrak{M}_l=
\begin{pmatrix}
A_{11}^{(l)} & \hdots & A_{1d}^{(l)}\\
\vdots & \ddots & \vdots\\
A_{d1}^{(l)} & \hdots & A_{dd}^{(l)}
\end{pmatrix}\quad\textrm{and}\quad \mathfrak{g}_l=\mathfrak{M}_l\begin{pmatrix}
1\\
& a^{n_1}\\
 & & a^{n_2}\\
& & & \ddots\\
& & & & a^{n_{d-1}} 
\end{pmatrix}
\end{equation}
for $l=0,\dots,n$. We further consider the ideal
\begin{equation}
\tilde{\mathfrak{I}}_{\cA,a,n}=I_2\begin{pmatrix}
\mathfrak{g}_0\begin{pmatrix}
x_{10}\\
\vdots\\
x_{d0}
\end{pmatrix} &
\cdots
& \mathfrak{g}_n\begin{pmatrix}
x_{1n}\\
\vdots\\
x_{dn}
\end{pmatrix}
\end{pmatrix}\subset \mathfrak{A}_{\cA}\left[(x_{ij})_{\substack{i=1,\dots,d\\j=1,\dots,n}}\right]
\end{equation}
and define
\begin{equation}
\label{equ-idgen}
\mathfrak{I}_{\cA,a,n}\coloneqq\mathrm{sat}(\tilde{\mathfrak{I}}_{\cA,a,n},a).
\end{equation}
Let $G=(f_1,\dots,f_n)$ be a generating set of $\mathfrak{I}_n$ and consider
\begin{equation}
[f_i]\equiv f_i\,\mathrm{mod}\, a.
\end{equation}
Then $\{[f_1],\dots,[f_n]\}$ generates $\faktor{\mathfrak{I}_{\cA,a,n}}{(a)}$.

Let $\underline{a}=\left(a_{ij}^{(l)}\right)_{\substack{i,j=1,\dots,d\\l=0,\dots,n}}$ with $a_{ij}^{(l)}\in \cA$ and recall the homomorphism
\begin{equation}
\mathrm{Subst}_{\underline{a}}:\mathfrak{A}_{\cA}[(x_{ij})]\to \cA[(x_{ij})]
\end{equation}
induced by $A_{ij}^{(l)}\mapsto a_{ij}^{(l)}$. Further, let $\tilde{G}=(f_1,\dots,f_n)$ be a Groebner basis of $(\tilde{\mathfrak{I}},1-ay)$ and $\underline{a}$, such that \cref{lem-redcrit} applies. Then, we have $G=\tilde{G}\cap\mathfrak{A}_{\cA}[(x_{ij})])$ is a Groebner basis of $\mathfrak{I}_{\cA,a,n}$  and according to \cref{lem-redcrit}, we have $\mathrm{Subst}_{\underline{a}}(G)$ is a Groebner basis of 
\begin{equation}
\label{equ-genred}
\mathrm{Subst}_{\underline{a}}(\mathfrak{I}_{\cA,a,n})=\mathrm{sat}(\mathrm{Subst_{\underline{a}}}(\tilde{\mathfrak{I}}_{\cA,a,n}),\mathrm{Subst}_{\underline{a}}(a)).
\end{equation}
Recall that for $G=(h_1,\dots,h_m)$, $\tilde{h}_i=\mathrm{Subst}_{\underline{a}}(h_i)$ and $[\tilde{h}_i]\equiv \tilde{h}_i\,\mathrm{mod}\,\mathrm{Subst}_{\underline{a}}(a)$, we have $\left([\tilde{h}_1],\dots,[\tilde{h}_m]\right)$ generates $\faktor{\mathrm{Subst}_{\underline{a}}(\mathfrak{I})_{\cA,a,n}}{\left(\mathrm{Subst}_{\underline{a}}(a)\right)}$. In other words, we have
\begin{equation}
\mathrm{Subst}_{\underline{a}}\left(\faktor{\mathfrak{I}_{\cA,a,n}}{(a)}\right)=\faktor{\mathrm{Subst}_{\underline{a}}(\mathfrak{I}_{\cA,a,n})}{\left(\mathrm{Subst}_{\underline{a}}(a)\right)}=\faktor{\mathrm{sat}(\mathrm{Subst_{\underline{a}}}(\tilde{\mathfrak{I}}_{\cA,a,n}),\mathrm{Subst}_{\underline{a}}(a))}{\left(\mathrm{Subst}_{\underline{a}}(a)\right)}.
\end{equation}

For the rest of this section, we fix a non-archimedean field $K$ with ring of integers $\mathcal{O}_K$, uniformiser $\pi$ and residue field $k=\faktor{\mathcal{O}_K}{(\pi)}$. 

\begin{proposition}
\label{prop-submust}
There exists a general subset $U\subset\mathcal{O}_K^{d^2(n+1)}$, such that we have $\underline{a}\in\mathcal{O}_{K}^{d^2\cdot(n+1)}$ for all $\underline{a}\in U$. Then, we have
\begin{equation}
\mathrm{Subst}_{\underline{a}}(\mathfrak{I}_{O_{K},a,n})=I(\mathcal{M}(\Gamma_{\underline{a},\underline{n}})).
\end{equation}
In particular, by \cref{equ-genred} we have for $\underline{a}\in U$ that
\begin{equation}
\mathrm{Subst}_{\underline{a}}(\faktor{\mathfrak{I}_{O_{K},a,n}}{(\pi)})=I(\mathcal{M}(\Gamma_{\underline{a},\underline{n}})_k).
\end{equation}
\end{proposition}

\begin{proof}
We first recall that the ideal $I(\mathcal{M}(\Gamma_{\underline{a},\underline{n}}))$ may be viewed as the saturation of the ideal generated by the $2\times2$ minors  of
\begin{equation}
\label{equ-matrix}
\begin{pmatrix}
g_0\begin{pmatrix}
x_{10}\\
\vdots\\
x_{d0}
\end{pmatrix}
&
\hdots
&
g_n\begin{pmatrix}
x_{1n}\\
\vdots\\
x_{dn}
\end{pmatrix}
\end{pmatrix}
\end{equation}

over $\mathcal{O}_K$ with respect to the uniformiser $\pi$. In other words, denoting the $2\times2$ minors of this matrix by $f_1,\dots,f_{\gamma}$, we have
\begin{equation}
I(\mathcal{M}(\Gamma_{\underline{a},\underline{n}}))=\mathrm{sat}(\langle f_1,\dots,f_{\gamma}\rangle_{\mathcal{O}_K[(A_{ij}^{(l)})]},\pi).
\end{equation}

Furthermore, we denote the $2\times2$ minors of
\begin{equation}
\begin{pmatrix}
\mathfrak{g}_0\begin{pmatrix}
x_{10}\\
\vdots\\
x_{d0}
\end{pmatrix}
&
\hdots
&
\mathfrak{g}_n\begin{pmatrix}
x_{1n}\\
\vdots\\
x_{dn}
\end{pmatrix}
\end{pmatrix}
\end{equation}

by $(\mathfrak{f}_1,\dots,\mathfrak{f}_\gamma)$ and observe that -- up to reordering -- that $\mathrm{Subst}_{\underline{a}}(\mathfrak{f}_i)=f_i$. Therefore, we have
\begin{equation}
\mathrm{Subst}_{\underline{a}}(\mathfrak{I}_{\mathcal{O}_{K,\pi,n}})=\mathrm{Subst}_{\underline{a}}\left(\mathrm{sat}(\langle\mathfrak{f}_1,\dots,\mathfrak{f}_\gamma\rangle_{\mathcal{O}_{K}[(A_{ij}^{(l)}]},\pi)\right).
\end{equation}

 According to \cref{lem-redcrit}, there are finitely many polynomials $(\mathfrak{P}_\alpha)_\alpha$ and $(\tilde{\mathfrak{P}}_\beta)_\alpha$ in $\mathcal{O}_K[(A_{ij}^{(l)})]$, where $\mathfrak{P}_\alpha$ is saturated with respect to $\pi$ and such that if $\mathrm{Subst}_{\underline{a}}(\mathfrak{P}_{\alpha})\in\mathcal{O}_K^\times$ and $\mathrm{Subst}_{\underline{a}}(\tilde{\mathfrak{P}}_\beta)\neq0$, then 
 
 \begin{equation}
 \label{equ:relevantstep}
 \mathrm{Subst}_{\underline{a}}\left(\mathrm{sat}(\langle\mathfrak{f}_1,\dots,\mathfrak{f}_\gamma\rangle_{\mathcal{O}_{K_n}[(A_{ij}^{(l)}]},\pi)\right)=\mathrm{sat}\left(\mathrm{Subst}_{\underline{a}}(\langle\mathfrak{f}_1,\dots,\mathfrak{f}_\gamma\rangle_{\mathcal{O}_{K_n}[A_{ij}^{(l)}]},\pi\right).
 \end{equation}
 
However, as all $\mathfrak{P}_{\alpha}$ are saturated with respect to $\pi$, we have

\begin{equation}
\mathfrak{P}_{\alpha}\,\mathrm{mod}\,\pi\not\equiv0
\end{equation}
for all alpha. Therefore, $\mathrm{Subst}_{\underline{a}}(\mathfrak{P}_{\alpha})\in\mathcal{O}_{K_n}^\times$ is a general condition. In other words, the set of $\underline{a}$, such that $\mathfrak{P}_{\alpha}(\underline{a})\,\mathrm{mod}\,\pi\not\equiv0$ is a non-empty general subset $U_{\alpha}$ of $\mathcal{O}_K^{d^2(n+1)}$.
 Moreover, the condition that $\tilde{\mathfrak{P}}_{\beta}\neq0$ for all $\beta$ is a generic condition over $K$. As noted in \cref{rem:general}, any non-empty generic set over $K_n$ contains a non-empty general subset. Thus, for any $\beta$, there exists a non-empty general subet $\tilde{U}_{\beta}$ of $\mathcal{O}_K^{d^2(n+1)}$, such that $\tilde{\mathfrak{P}}_\beta(\underline{a})\neq0$ for all $\underline{\beta}\in \tilde{U}_{\beta}$. As the intersection of finitely many non-empty general subsets is non-empty, we have that $\mathrm{Subst}_{\underline{a}}(\mathfrak{P}_{\alpha})\in\mathcal{O}_{K}^\times$ and  $\mathrm{Subst}_{\underline{a}}(\tilde{\mathfrak{P}}_{\beta})\neq0$ are true for $\underline{a}$ in in the non-empty general subset $U=\left(\bigcap U_{\alpha}\right)\cap\left(\bigcap \tilde{U}_\beta\right)$. Therefore, \cref{equ:relevantstep} holds. We may summarise the above considerations in the following equation
\begin{align}
&\mathrm{Subst}_{\underline{a}}(\mathfrak{I}_{\mathcal{O}_{K_n,\pi,n}})=\mathrm{Subst}_{\underline{a}}\left(\mathrm{sat}(\langle\mathfrak{f}_1,\dots,\mathfrak{f}_\gamma\rangle_{\mathcal{O}_{K_n}[(A_{ij}^{(l)}]},\pi)\right)=\mathrm{sat}\left(\mathrm{Subst}_{\underline{a}}(\langle\mathfrak{f}_1,\dots,\mathfrak{f}_\gamma\rangle_{\mathcal{O}_{K_n}[A_{ij}^{(l)}]},\pi\right)\\
&=\mathrm{sat}\left(\langle f_1,\dots,f_\gamma\rangle_{\mathcal{O}_{K_n}},\pi\right)=I(\mathcal{M}(\Gamma_{\underline{a},\underline{n}})),
\end{align}
which completes the proof.
\end{proof}

Let $y_1,\dots,y_d$ be the coordinates of $\mathbb{P}(V)$ and let $X\subset\mathbb{P}(V)$ be a geometrically irreducible variety with $X=V(I')$ for an ideal $I'\subset K[y_i]$, such that $\mathrm{dim}(X)=\mathrm{dim}(X\times_K\mathrm{Spec}(\overline{K}))$, i.e. $X$ is dense in its base change to the algebraic closure of $K$.\\
We consider the map
\begin{equation}
\tilde{f}_{\Gamma_{\underline{a},\underline{n}}}:\mathbb{P(V)}\xrightarrow{g_0^{-1}\times\dots\times g_n^{-1}}\mathbb{P}(V)^{n+1}.
\end{equation}
We may compute $\overline{\tilde{f}_{\Gamma_{\underline{a},\underline{n}}}(X)}$ as follows. Let $(\alpha_{i})_{i=0,\dots,n}$ be a vector of formal variables, $m_{ij}$ the $i-$th entry of $g_j^{-1}\cdot\underline{y}$ and set the ideal $I''=\left\langle I, (\alpha_{j}x_{ij}-m_{ij})_{i,j}\right\rangle_{K[x_{ij},\alpha_{i},y_l]}$. Then, it is well-known that -- since $X$ is Zariski dense in its pullback to the algebraic closure of $K$ -- for 
\begin{equation}
\label{must:deg}
\tilde{I}=(\mathrm{sat}(I''\cap K[x_{ij},\alpha_i]),\langle (\alpha_i)_i\rangle)\cap K[x_{ij}], 
\end{equation}
we have $V(\tilde{I})=\overline{\tilde{f}_{\Gamma_{\underline{a},\underline{n}}}(X)}$. Furthermore, we may choose  a set of generators of $(h_1,\dots,h_\nu)$ of $\tilde{I}$, such that $h_i\in\mathcal{O}_K[x_{ij}]$. Then, we define $I=\mathrm{sat}(\langle h_1,\dots,h_\nu\rangle,\pi)$, observe that $I=\tilde{I}\cap \mathcal{O}_K[x_{ij}]$ and thus we obtain that $I$ and $I(X(\Gamma_{\underline{a},\underline{n}}))$ define the same topological subspace of $\mathbb{P}(L)^{n+1}$.\\
Now, let $\mathfrak{m}_{ij}$ denote the $i-$th entry of $\mathfrak{g}_j^{-1}\cdot\underline{y}$ and set the ideal $\mathfrak{L}'=\left\langle I', (\alpha_jx_{ij}-\mathfrak{m}_{ij})_{i,j}\right\rangle_{K[A_{ij}^{(l)}][x_{ij},y_l]}$. We then set 
\begin{equation}
\label{must:deg1}
\tilde{\mathfrak{L}}=(\mathrm{sat}(\mathfrak{L}'\cap K[A_{ij}^{(l)}][x_{ij},\alpha_i]),\langle (\alpha_i)_i\rangle)\cap K[A_{ij}^{(l)}][x_{ij}].
\end{equation}
We may choose generators $\mathfrak{h}_1,\dots,\mathfrak{h}_\gamma$ of $\tilde{\mathfrak{L}}$, such that $\mathfrak{h}_i\in \mathcal{O}_K[A_{ij}^{(l)}][x_{ij}]$ and we define
\begin{equation}
\label{equ:key}
\mathfrak{L}=\mathrm{sat}(\langle \mathfrak{h}_1,\dots,\mathfrak{h}_\gamma\rangle,\pi.)
\end{equation}

To end this subsection, we prove the following corollary.

\begin{corollary}
\label{cor-subvar}
Let $X$ be a geometrically irreducible variety, let $\tilde{I}$ be the associated ideal in \cref{must:deg} and $\mathfrak{L}$ the associated ideal in \cref{must:deg1}. There exists a general subset $U\subset\mathcal{O}_K^{d^2(n+1)}$, such that $\mathrm{Subst}_{\underline{a}}(\mathfrak{L})$ and $I(X(\Gamma_{\underline{a},\underline{n}}))$ cut out the same topological subspace of $\mathbb{P}(L)^n$ for all $\underline{a}\in U$. In particular, we have that $\mathrm{Subst}_{\underline{a}}(\mathfrak{L})\,\mathrm{mod}\,\pi$ cuts out the topological space underlying $X(\Gamma_{\underline{a},\underline{n}})_k$.
\end{corollary}

\begin{proof}
We first observe that $\mathrm{Subst}_{\underline{a}}(\mathfrak{L}')=I''$ for all $\underline{a}$.  The corollary follows when we prove that $\mathrm{Subst}_{\underline{a}}(\mathfrak{L})=I$ for a general choice of $\underline{a}$. However, all steps involved in the definitions of $\mathfrak{L}$ and $I$ may be resolved by Groebner basis computations. Therefore, the same arguments as in the proof of \cref{prop-submust} yield the desired result.
\end{proof}

\subsection{Proof of \cref{thm-main1}}
\label{sec:proof}
We are now ready to give the proof of \cref{thm-main1}. First, we let $K$ be equal to its maximal unramified field extension. Therefore, since $k$ is a perfect field, we have $k$ algebraically closed. We further assume that \cref{conj-must} holds. Let $U_{\underline{a},n}$ be the general subset of $\mathcal{O}_K^{d^2(n+1)}$ in \cref{conj-must}, i.e. $I(\mathcal{M})(\Gamma_{\underline{a},\underline{n}})=\bigcap I_v$ for all $\underline{a}\in U_{\underline{n},n}$. Let $X$ be a projective variety as stated in the theorem. Furthermore, let $V_n$ the general subset in \cref{cor-subvar}, i.e. $\mathrm{Subst_{\underline{a}}}(\mathfrak{L})=I(X(\Gamma_{\underline{a},\underline{n}}))$ for all $\underline{a}\in V_n$. In particular, we have $\mathrm{Subst_{\underline{a}}}(\mathfrak{L})\,\mathrm{mod}\,\pi=I(X(\Gamma_{\underline{a},\underline{n}})_k)$ for all $\underline{a}\in V_n$.\\
Recall that since we assumed that \cref{conj-must} holds, we have for all $\underline{a}\in U_{\underline{n},n}$ that
\begin{equation}
\mathcal{M}(\Gamma_{\underline{a},\underline{n}})_k=\bigcap I_v,
\end{equation}
where the intersection runs over all $v=(v_0,\dots,v_n)\in\mathbb{Z}^{n+1}$, such that $\sum v_i=n(d-1)$ and $0\le v_i\le d-1$. For such $\underline{a}$, we have
\begin{equation}
\mathcal{M}(\Gamma_{\underline{a},\underline{n}})_{k,\le l}=\bigcap_{v\in Y_l} I_v,
\end{equation}
where $Y_l$ is the set of all $v=(v_0,\dots,v_n)\in\mathbb{Z}^{n+1}$, such that $\sum v_i=n(d-1)$, $0\le v_i\le d-1$ and $\left|\{i\in\{0,\dots,n\}\mid v_i<d-1\}\right|\le l$. We note that $\mathcal{M}(\Gamma_{\underline{a},\underline{n}})_{k}=\bigcap_{Y_{d-1}} I_v$.\\
We denote $\mathfrak{L}_k=\mathfrak{L}\,\mathrm{mod}\,\pi$. Let $f_1^{(l)},\dots,f_{t_l}^{(l)}\in$ be generators of $\bigcap_{Y_l} I_v$ and let $G_i^{(l)}$ be a Groebner basis of $\langle\mathcal{L},1-yf_i^{(l)}\rangle_{\mathcal{O}_K[A_{ij}^{(l')}]}$ with respect to some order $<$, such that $y$ is the maximal element among $\{y,(x_{ij})_{i,j})\}$. Now, let $\delta_n=\mathrm{min}_{l}((G_i^{(l)}\cap k[A_{ij}^{(l)}]\neq(0)\,\textrm{for all}\, i)$.\\
As $k$ is algebraically closed, we have $X(\Gamma_{\underline{a},\underline{n}})_k\subset\bigcap_{Y_l}V(I_v)$ is equivalent to $1\in \langle I(X(\Gamma_{\underline{a},\underline{n}})_k),1-yf_i^{(l)}\rangle$ for all $i$, since this is equivalent to $f_i\in\sqrt{I(X(\Gamma_{\underline{a},\underline{n}})_k)}=I(X(\Gamma_{\underline{a},\underline{n}})_k^{\textrm{red}})$, where $X(\Gamma_{\underline{a},\underline{n}})_k^{\textrm{red}}$ is the unique reduced scheme on the topological space underlying $X(\Gamma_{\underline{a},\underline{n}})_k$. Moreover, by the same argument as in the proof of \cref{prop-submust}, we have that there exists a general subset $V'^{(l)}_n\subset \mathcal{O}^{d^2(n+1)}$, such that $\mathrm{Subst}_{\underline{a}}(G_i^{(l)})$ is a Groebner basis of $\langle \mathrm{Subst}_{\underline{a}}(\mathfrak{L}_k),1-yf_i\rangle$ for all $\underline{a}\in V'^{(l)}_n$. Therefore, we have for all $\underline{a}\in V''^{(l)}_n=V\cap V'^{(l)}_n$ that $\mathrm{Subst}_{\underline{a}}(G_i^{(l)})$ is a Groebner basis of  $\langle I(X(\Gamma_{\underline{a},\underline{n}})_k),1-yf_i\rangle$ with respect to the restriction $<$ to $k[(x_{ij})_{i,j}]$. Thus, we have $1\in\langle I(X(\Gamma_{\underline{a},\underline{n}})_k),1-yf_i\rangle$ is equivalent to $\mathrm{Subst}_{\underline{a}}(G_i^{(l)})\cap k\neq (0)$. Furthermore, we observe that after possibly shrinking $V''^{(l)}_n$, we have that $G_i^{(l)}\cap k[A_{ij}^{(l)}]\neq(0)$ is equivalent to $\mathrm{Subst}_{\underline{a}}(G_i^{(l)})\cap k\neq (0)$. Let $\tilde{V}_n=\cap_l V''^{(l)}_n\cap U_{\underline{n},n}$. Therefore, we have 
\begin{equation}
X(\Gamma_{\underline{a},\underline{n}})_k\subset V(\bigcap_{v\in Y_l}I_v)\subset V(\bigcap_{Y_{d-1}}I_v)=\mathcal{M}(\Gamma_{\underline{a},\underline{n}})_k
\end{equation}
if and only if $G_i^{(l)}\cap k[A_{ij}^{(l')}]\neq(0)$ for all $\underline{a}\in\tilde{V}_n$. As $\tilde{V}_n$ is a finite intersection of non-empty general subsets, the set $\tilde{V}$ is a non-empty general subset of $\mathcal{O}_K^{d^2(n+1)}$ as well.\\
By the above considerations, we have also shown that
\begin{equation}
\delta_n=\mathrm{min}_{l}\left(X(\Gamma_{\underline{a},\underline{n}})_k\subset\bigcap_{v\in Y_l}I_v\right)
\end{equation}
for all $\underline{a}\in\tilde{V}_n$. Now, let $v\in Y_{\delta_n}\backslash Y_{\delta_n-1}$. Then, we have $V(I_v)\subset\mathcal{M}(\Gamma_{\underline{a},\underline{n}})_{k,\le \delta_n}$ but $V(I_v)\subset\mathcal{M}(\Gamma_{\underline{a},\underline{n}})_{k,\le \delta_n-1}$ for all $\underline{a}\in U_{\underline{n},n}$. Recall that $f_1^{(\delta_n-1)},\dots,f_{t_{\delta_n-1}}^{(\delta_n-1)}$ is a set of generators of $\bigcap_{Y_{\delta_n-1}}I_v$. Then, we have that $X(\Gamma_{\underline{a},\underline{n}})_k\cap V(I_v)\subset\mathcal{M}(\Gamma)_{k,\le l-1}$ if and only if we have $1\in \langle I(X(\Gamma_{\underline{a},\underline{n}})_k,I_v,1-yf_i^{(\delta_n-1)}\rangle$ for all $i$. Let $G_i^{(\delta_n-1),v}$ be a Groebner basis of $\langle\mathcal{L},I_v,1-yf_i^{(\delta_n-1)}\rangle$. By the same argument as above, after possibly shrinking $\tilde{V}_n$ we have $1\in \langle I(X(\Gamma_{\underline{a},\underline{n}})_k,I_v,1-yf_i^{(\delta_n-1) }\rangle$ if and only if $G_i^{(\delta_n-1),v}\cap [A_{ij}^{(l)}]\neq (0)$ for all $\underline{a}\in\tilde{V}_n$. As for all $\underline{a}\in\tilde{V}_n$, we have that
\begin{equation}
X(\Gamma_{\underline{a},\underline{n}})_k\subset V(\bigcap_{Y_{\delta_n}}I_v)=\bigcup_{Y_{\delta_n}}V(I_v)
\end{equation}
but 
\begin{equation}
X(\Gamma_{\underline{a},\underline{n}})_k\not\subset V(\bigcap_{Y_{\delta_n-1}}I_v)=\bigcup_{Y_{\delta_n-1}}V(I_v).
\end{equation}
Therefore, there exists an irreducible component $\tilde{X}$ of $X(\Gamma_{\underline{a},\underline{n}})_k$ and $v\in Y_{\delta_n}\backslash Y_{\delta_n-1}$, such that $\tilde{X}\subset V(I_v)$ but $\tilde{X}\not\subset\bigcup_{Y_{\delta_n-1}}V(I_v)$. In particular, we have $\tilde{X}\cap V(I_v)\not\subset\bigcup_{Y_{\delta_n-1}}V(I_v)$. Thus, there exists $\ell\in\{1,\dots,t_{\delta_n-1}\}$, such that $\mathrm{Subst_{\underline{a}}}(G_\ell^{(\delta_n-1}),v)\cap k=(0)$. Since $\underline{a}\in\tilde{V}_n$, we have $G_{\ell}^{(\delta_n-1),v}\cap k[A_{ij}^{(l)}]=(0)$.\\
Now, let $\sigma$ be a permutation on $\{0,\dots,n\}$. Then $\sigma$ acts naturally on the sets $Y_l$ by permuting the entries of the $v\in Y_l$. Moreover, we consider the ring automorphism
\begin{equation}
\tau_{\sigma}:k[A_{ij}^{(l)}][x_{ij}]\to k[A_{ij}^{(l)}][x_{ij}]
\end{equation}
induced by $x_{ij}\mapsto x_{i\sigma(j)}$ and $A_{ij}^{(l)}\mapsto A_{ij}^{\sigma(l)}$. Then, for any $\sigma$, we have $\tau_{\sigma}(I_v)=I_{\sigma(v)}$. Moreover, by construction, we have $\tau_{\sigma}(\mathfrak{L}_k)=\mathfrak{L}_k$ for any $\sigma$. We note that as $\sigma(\bigcap_{Y_{\delta_n-1}}I_v)=\bigcap_{Y_{\delta_n-1}}I_v$, we have $\sigma(f_{\ell}^{\delta_n-1})\in\bigcap_{Y_{\delta_n-1}}I_v$. Moreover, we have $\sigma(G_{\ell}^{\delta_n-1})$ is a Groebner basis of $\langle\mathfrak{L}_k,I_{\sigma(v)},1-y\sigma(f_{\ell}^{\delta_n-1})\rangle$ and $(0)=\sigma(G_{\ell}^{(\delta_n-1),v}\cap k[A_{ij}^{(l)}])=\sigma(G_{\ell}^{\delta_n-1})\cap k[A_{ij}^{(l)}]$. Therefore, after possibly shrinking $\tilde{V}_n$, we have $1\notin \langle I(X(\Gamma_{\underline{a},\underline{n}})_k),I_{\sigma(v)},1-y\sigma(f_{\ell}^{\delta_n-1})\rangle$. Thus, we have 
\begin{equation}
X(\Gamma_{\underline{a},\underline{n}})_k\cap V(I_{v})\not\subset V(\bigcap_{w\in Y_{\delta_n-1}}I_w)\quad\textrm{and}\quad X(\Gamma_{\underline{a},\underline{n}})_k\cap V(I_{\sigma(v)})\not\subset V(\bigcap_{w\in Y_{\delta_n-1}}I_w)
\end{equation}
for all $\underline{a}\in\tilde{V}_n$. Inductively, we obtain that after possibly shrinking $\tilde{V}_n$ again, we have
\begin{equation}
\label{equ-subset}
X(\Gamma_{\underline{a},\underline{n}})_k\cap V(I_{\sigma(v)})\not\subset V(\bigcap_{w\in Y_{\delta_n-1}}I_w)=\bigcup_{w\in Y_{\delta_n-1}}V(I_w)
\end{equation}
for all permutations $\sigma$ and all $\underline{a}\in\tilde{V}$.\\
Our next goal is to show that for each $Z\subset\{0,\dots,n\}$ with $|Z|=\delta_n$, we have that there exists an irreducible component $X_Z$ of $X(\Gamma_{\underline{a},\underline{n}})_k$, such that
\begin{equation}
X_Z\subset\bigcup_{\substack{v\textrm{, s.th.}\\\mathrm{supp}(V(I_v))=Z}}V(I_v)
\end{equation}
for all $\underline{a}\in \tilde{V}_n$. By \cref{equ-subset}, there exists an irreducible component $\tilde{X}$ of $X(\Gamma_{\underline{a},\underline{n}})_k$, such that $\tilde{X}\subset\bigcup_{v\in Y_{\delta_n}}V(I_v)$ and 
\begin{equation}
\tilde{X}\cap\bigcup_{\substack{v\textrm{, s.th.}\\\mathrm{supp}(V(I_v))=Z}}V(I_v)\not\subset\bigcup_{v\in Y_{\delta_n-1}}V(I_v).
\end{equation}
If 
\begin{equation}
\tilde{X}\not\subset\bigcup_{\substack{v\textrm{, s.th.}\\\mathrm{supp}(V(I_v))=Z}}V(I_v),
\end{equation}
 there exists $w\in Y_{\delta_n}$ with $\mathrm{supp}V(I_w)\neq Z$ with $\tilde{X}\subset V(I_w)$. However, this yields 
\begin{align}
&\tilde{X}\cap\bigcup_{\substack{v\textrm{, s.th.}\\\mathrm{supp}(V(I_v))=Z}}V(I_v)=(\tilde{X}\cap V(I_w))\cap\bigcup_{\substack{v\textrm{, s.th.}\\\mathrm{supp}(V(I_v))=Z}}V(I_v)\\
&=\tilde{X}\cap \left(V(I_w)\cap\bigcup_{\substack{v\textrm{, s.th.}\\\mathrm{supp}(V(I_v))=Z}}V(I_v)\right)\subset \left(V(I_w)\cap\bigcup_{\substack{v\textrm{, s.th.}\\\mathrm{supp}(V(I_v))=Z}}V(I_v)\right)\subset \bigcup_{v\in Y_{\delta_n-1}}V(I_v).
\end{align}
This contradics \cref{equ-subset}, which proves the claim that for all $\underline{a}\in\tilde{V}_n$ and for any $Z\subset\{0,\dots,n\}$ with $|Z|=\delta_n$, there exists an irreducible component $X_Z$ of $X(\Gamma_{\underline{a},\underline{n}})_k$, such that
\begin{equation}
X_Z\subset\bigcup_{\substack{v\textrm{, s.th.}\\\mathrm{supp}(V(I_v))=Z}}V(I_v).
\end{equation}
Thus, for all $\underline{a}\in\tilde{V}_n$ there exist at least
\begin{equation}
\binom{n}{\delta_n}
\end{equation}
irreducible components of $X(\Gamma_{\underline{a},\underline{n}})_k$.\\
If for all $n$, we have $\delta_n\le d'$, we are done. Now, assume there exists $n_0$, such that $\delta_{n_0}>d'$.  In order to reach a contradiction, we first observe that if $n_0\gg0$, we have
\begin{equation}
\label{equ-binom}
\binom{n_0}{\delta_{n_0}}>\mathrm{deg}(X)\binom{n_0+d'}{d'}.
\end{equation}
Therefore, the number of irreducible components of $X(\Gamma_{\underline{a},\underline{n}})_k$ would exceed the upper bound in \cref{lem-upbound}, which yields a contradiction. Therefore, for all sufficiently large $n$ the theorem holds. Now, let $n_0$ be arbitrary and $n_1$, such that \cref{equ-binom} holds. We consider the projection $\mathrm{Pr}:\mathcal{O}_K^{d^2(n_1+1}\to\mathcal{O}_K^{d^2(n_0+1)}$ mapping $(a_{ij}^{(l)})_{\substack{i,j=1,\dots,d\\l=0,\dots,n_1}}$ to $(a_{ij}^{(l)})_{\substack{i,j=1,\dots,d\\l=0,\dots,n_0}}$. Then $\mathrm{Pr}^{-1}(\tilde{V}_{n_0})$ is a general subset of $\mathcal{O}_K^{d^2(n_1+1}$. Then, we have already shown that for all $\underline{a}\in\tilde{V}_{n_1}\cap\mathrm{Pr}^{-1}(\tilde{V}_{n_0})$ that $\mathcal{X}(\Gamma_{\underline{a},\underline{n}})\subset\mathcal{M}(\Gamma_{\underline{a},\underline{n}})_{k,\le d'}$. We see that under the projection
\begin{equation}
pr_{n_1,n_2}:\mathbb{P}(L)^{n_1+1}\to\mathbb{P}(L)^{n_0+1}
\end{equation}
onto the first $n_0+1$ factors, we have that $pr_{n_1,n_2}(\mathcal{M}(\Gamma_{\underline{a},\underline{n}})_{k,\le d'})=\mathcal{M}(\Gamma_{\mathrm{Pr}(\underline{a}),\underline{n}})_{k,\le d'}$ and $pr_{n_1,n_2}(X(\Gamma_{\underline{a}},\underline{n}))=X(\Gamma_{\mathrm{Pr}(\underline{a}}),\underline{n})$. Thus, we have $X(\Gamma_{\mathrm{Pr}(\underline{a}),\underline{n}})\subset\mathcal{M}(\Gamma_{\mathrm{Pr}(\underline{a}),\underline{n}})_{k,\le d'}$ for all $\underline{a}\in \tilde{V}_{n_1}\cap\mathrm{Pr}^{-1}(\tilde{V}_{n_0})$. Therefore, we have
\begin{equation}
X(\Gamma_{\underline{a},\underline{n}})\subset\mathcal{M}(\Gamma_{\mathrm{Pr}(\underline{a}),\underline{n}})_{k,\le d'}
\end{equation}
for all $\underline{a}\in \tilde{\tilde{V}}_{n_0}\subset\mathrm{Pr}(\tilde{V}_n\cap\mathrm{Pr}^{-1}(\tilde{V}_{n_0})$. As projection is an open morphism, we have that $\tilde{\tilde{V}}_{n_0}$ is a general subset of $\mathcal{O}_K^{n_0+1}$, which finishes the proof when $k$ is algebraically closed.\\
Now, let $K$ be any discretely valuated field with perfect residue field $k$. We note that when passing to the maximal unramified field extension $K^{un}$ of $K$, we have that the residue field of $K^{un}$ is algebraically closed. We have thus already proved that there exists an general subset $U_n\subset\mathcal{O}_{K^{un}}^{d^2(n+1)}$, such that $X(\Gamma_{\underline{a},\underline{n}})_k\subset\mathcal{M}(\Gamma_{\underline{a},\underline{n}})_{k,\le d'}$. Moreover, we see immediately that for all $\underline{a}$ contained in the general subset $U_n\cap \mathcal{O}_K^{d^2(n+1)}$ the same as true.  However, the set $ U_n\cap \mathcal{O}_K^{d^2(n+1)}$ might be empty. Since there exists a finite unramified field extension $K'$ of $K$, such that $U_n\cap \mathcal{O}_{K'}^{d^2(n+1)}\neq\emptyset$ the theorem follows.
\qed

\subsection{Models of syzygy bundles on curves}
\label{sec:models}
After studying degenerations of projective varieties, we will now turn our focus to degenerations of syzygy bundles on projective curves. Recall that $K$ is a discretely valued field with ring of integers $\mathcal{O}_K$ and residue field $k$.\par 
We fix two positive integers $n\geq 2$ and $\rho$ and  non-negative numbers $d_0,\dots,d_{n}\le \rho$ with $\sum_{j =0}^{n} d_j=n\rho$. Let $F_0,\dots,F_{n}$ be polynomials with
\begin{equation}
F_i\in\mathrm{Sym}_{j\neq i}\mathcal{O}_K[x_{1j},\dots,x_{dj}]^{(\rho-d_j)},
\end{equation}
where $\mathcal{O}_K[x_{1j},\dots,x_{dj}]^{(\rho-d_j)}$ denotes the $\mathcal{O}_K$-submodule of the polynomial ring consisting of all homogenous polynomials of degree $\rho - d_j$. Hence $F_i$ is a linear combination of products of $n$ homogenous polynomials, each in a different set of variables. 

For a fixed choice 
\begin{equation}
\underline{a}=\left(a_{ij}^{(l)}\right)_{\substack{i,j=1,2,3\\l=0,\dots,n}}\in \mathcal{O}_K^{9(n+1)},
\end{equation}
such that all $g_i$ are invertible and a projective curve $X\subset\mathbb{P})V)$, let $X(\Gamma_{\underline{a},\underline{n}})$ be the associated Mustafin model of $X$. Now, we fix $d_i$ as above and consider the $\mathcal{O}_K$-linear ring homomorphism

\begin{align}
\label{equ:morphwell}
\begin{split}
\Upsilon_i^{\underline{a}}:\mathrm{Sym}_{j\neq i}\mathcal{O}_K[x_{1j},\dots,x_{dj}]^{(\rho-d_j)}&\to K[x_1,\dots,x_d]^{(d_i)}
\end{split}
\end{align}

induced by

\begin{equation}
\begin{pmatrix}
x_{1j}\\
\vdots\\
x_{dj}
\end{pmatrix}
\mapsto g_j^{-1}
\begin{pmatrix}
x_1\\
\vdots\\
x_d
\end{pmatrix}.
\end{equation}
Note that \cref{equ:morphwell} is well-defined due to the fact that $\sum_{j=1}^{n+1} d_j=n\rho$, which yields $d_i=\sum_{j\neq i}(\rho-d_j)$.

 We denote by $\Sigma_i$ the subset of $\mathrm{Sym}_{j\neq i}\mathcal{O}_K[x_{1j},\dots,x_{dj}]^{(\rho-d_j)}$ of polynomials $F$, such that the saturation $F'$ reduces to a polynomial $\overline{F}'$ modulo $\mathfrak{m}$ which satisfies
\begin{equation}
\overline{F}'\not\in \langle (x_{1j},\dots,x_{(d-1)j})_{j\neq i}\rangle.
\end{equation}

\begin{definition}
Let $f_0,\dots,f_{n}$ be polynomials with $f_i \in K[x_1,\dots,x_d]^{(d_i)}$ with degrees $d_i$ as above. Then we  say that the tuple $(f_0,\dots,f_{n})$ is \textit{$(\underline{d},\underline{a})-$admissible} if $f_i\in\Upsilon_i^{\underline{a}}(\Sigma_i)$.
\end{definition}

We are now ready to state our second main result.

\begin{proposition}
\label{thm:proj}
We fix natural numbers $n\ge2$ and $\rho$ and non-negative integers $d_0,\dots,d_{n}\leq \rho$, such that $\sum d_i=n\rho$. Furthermore, let $f_0,\dots,f_{n}$ be polynomials with $f_i \in K[x_1,\dots,x_d]^{(d_i)}$. \par 
Let $X\subset\bigcup_{i=0}^{n} D_+(f_i)\subset\mathbb{P}(V)$ be a smooth projective curve, and let $\underline{a} \in \mathcal{O}_K^{9(n+1)}$ be a choice of coefficients, such that all $g_i$ are invertible, $X(\Gamma_{\underline{a},\underline{n}})$ has star-like reduction and such that $(f_0,\dots,f_{n})$ is $(\underline{d},\underline{a})-$admissible.\par
Then there exists a vector bundle $\cE$  on $X(\Gamma_{\underline{a},\underline{n}})$ with generic fibre $E=\restr{\mathrm{Syz}(f_0,\dots,f_{n})(\rho)  }{C}$ whose special fibre is trivial on all reduced irreducible components of $X(\Gamma_{\underline{a},\underline{n}})_k$. 
\end{proposition}

\begin{remark}
As explained in \cref{sec:semi}, this  implies that the bundle $E=\restr{\mathrm{Syz}(f_1,\dots,f_{n+1})(\rho)  }{C}$ is semistable of degree $0$ on $C$.
\end{remark}

\begin{proof}
The proof is analogous to \cite[theorem 4.3]{HWplane}.
\end{proof}

This result immediately implies the following result.

\begin{corollary}
Assume that $K$ is contained in $\overline{\mathbb{Q}}_p$ and $n\ge2$. Let  $f_i \in K[x_1,\dots,x_d]^{(d_i)}$ be $n+1$ polynomials of homogenous degrees $d_i \leq \rho$ satisfying $\sum d_i=n\rho$. Consider a connected smooth plane curve $X$ contained in $\bigcup_{i=1}^{n+1} D_+(f_i)\subset\mathbb{P}(V)$ and assume that there exists a choice of matrix coefficients $\underline{a} \in \mathcal{O}_K^{9(n+1)}$, such that all $g_i$ are invertible, $X(\Gamma_{\underline{a},\underline{n}})$ has star-like reduction and such that $(f_1,\dots,f_{n+1})$ is $(\underline{d},\underline{a})-$admissible.\par
Then the base change of the syzygy bundle $E=\restr{\mathrm{Syz}(f_0,\dots,f_{n})(\rho)}{X}$  to  $\mathbb{C}_p$ has strongly semistable reduction in the sense of \cref{def:strongred}.
\end{corollary}

\bibliographystyle{alpha}
\bibliography{literature.bib}

\end{document}